\documentclass[12pt]{amsart}

\usepackage{amsthm}
\newtheorem{theorem}{Theorem}[section]
\newtheorem{lemma}[theorem]{Lemma}

\theoremstyle{definition}

\theoremstyle{remark}
\newtheorem{remark}[theorem]{Remark}

\counterwithin*{section}{part}

\usepackage{amssymb}
\usepackage{empheq}
\usepackage{stmaryrd}
\usepackage{url}
\usepackage{esvect}
\usepackage{float}

\DeclareMathOperator{\N}{\mathbb{N}}

\DeclareMathOperator{\R}{\mathbb{R}}
\DeclareMathOperator*{\argmin}{\arg\!\min} 

\newcommand{\norm}[1]{\lVert #1 \rVert}

\newcommand{\lowerbound}{\frac{1}{\gamma - 8 \eta}}
\newcommand{\upperbound}{\frac{1}{4\eta}}

\usepackage[left=2.0cm,%
                right=2.0cm,%
                top=2.5cm,%
                bottom=3.5cm,%
                headheight=12pt,%
                a4paper]{geometry}%

\begin{document}

\title[Relaxed-inertial proximal point algorithm on Hadamard manifolds]{A relaxed-inertial proximal point algorithm for strongly quasiconvex equilibrium problems on Hadamard manifolds}

\author[L.M. Despr\'es and N. Pischke]{Luisa Marie Despr\'es${}^{\MakeLowercase a}$ and Nicholas Pischke${}^{\MakeLowercase b}$}
\date{\today}
\maketitle
\vspace*{-5mm}
\begin{center}
{\scriptsize 
${}^a$ Department of Mathematics, Technische Universit\"at Darmstadt,\\
Schlossgartenstra\ss{}e 7, 64289 Darmstadt, Germany,\\ 
${}^b$ Department of Computer Science, University of Bath,\\
Claverton Down, Bath, BA2 7AY, United Kingdom,\\
E-mails: luisa.despres@stud.tu-darmstadt.de, nnp39@bath.ac.uk}
\end{center}

\maketitle
\begin{abstract}
We study a proximal point type method for approximating solutions to equilibrium problems generated by pseudomonotone and strongly quasiconvex bifunctions over Hadamard manifolds, that is complete simply connected Riemannian manifolds of nonpositive sectional curvature. Next to the usual proximal step, the method we consider also incorporates an inertia step together with a subsequent over-relaxation, the latter of which is treated in the context of Hadamard manifolds, to our knowledge, for the first time. Making use of a quantitative approach towards such proximal methods for strongly quasiconvex optimization developed by the authors in previous work, we in particular provide effective arguments for the convergence of the method, yielding explicit, fast and very uniform rates of convergence for the distance of the iterates towards the solution. These results extend previous work by Grad, Lara and Marcavillaca on such a method over finite-dimensional Euclidean spaces for the first time to a nonlinear setting, with the quantitative estimates already being novel in the Euclidean case. In particular, our effective approach allows for a fine-grained view on the assumptions on the surrounding objects, so that we are able to either weaken or even fully discharge some previous assumptions.
\end{abstract}

\section{Introduction}

\subsection{Background and motivation}

Many problems in applied mathematics, including in particular minimization problems, fixed point problems and variational inequalities, can be formulated as an equilibrium problem, that is as a problem of the form
\begin{equation*}
\text{find } x^{*}\in K\text{ such that } f(x^{*},y) \geq 0 \text{ for all } y\in K,\tag{EP}\label{equil}
\end{equation*}
for a suitable bifunction $f: K \times K \rightarrow \R$ and an associated constraint set $K \subseteq X$ in a space $X$. Also called ``Ky Fan inequalities'' after the pioneering work of Fan \cite{Fan72}, these problems have been studied deeply, and we refer to \cite{Blum1994,Cotrina18,Flores01,IusemKassaySosa09,IusemLara19,Lopez12,Oettli97} next to many other works, for such investigations.

Many prominent methods used to approximate solutions of convex equilibrium problems are based on (modifications of) the proximal point scheme. Originally developed for convex minimization problems and associated inclusion problems of maximally monotone operators by Martinet \cite{Martinet70}, Rockafellar \cite{RockTyr76} as well as Brezis and Lions \cite{BrezisLions78} over Hilbert spaces $(X,\norm{\cdot})$, this method iteratively applies the so-called proximal map 
\[
\mathrm{Prox}_{\beta h}(x) :=\argmin_{y\in K}\left\{h(y) + \frac{1}{2\beta}\norm{y-x}^2\right\}
\]
of a convex, lower-semicontinous function $h: K \rightarrow \R$ on a closed and convex subset $K\subseteq X$ to form the sequence $\{x_k\}_k$ defined by $x_{k+1}:=\mathrm{Prox}_{\beta_k h}(x_k)$, given some fixed starting point $x_0\in K$ and using an associated parameter sequence $\{\beta_k\}_k\subseteq (0,\infty)$. Under suitable assumptions on the parameters, $\{x_k\}_k$ weakly converges to a minimum of $h$ (see e.g.\ \cite{BauschCom2017}).

Many extensions of this proximal point type approach to (convex) equilibrium problems rely on utilizing a suitable analog of this proximal map defined by regularizing the whole bifunction, also called the resolvent of the bifunction (see already \cite{Moudafi99,MoudafiThera99} but also \cite{BURACHIK2012,IusemSosa10,Khatibzadeh2016,Konnov03,Moudafi03}). However, in \cite{IusemLara2021} (see also \cite{HieuDuongThai21,IusemMohebbi2020}), Iusem and Lara define an algorithm that in each iteration applies a proximal  regularization only to the second (convex) argument of the bifunction, that is it utilizes $\mathrm{Prox}_{\beta_kf(x_k,\cdot)}(x_k)$ to define a new iterate. Indeed, as illustrated in \cite{IusemLara2021}, this type of proximal step is particularly suitable for equilibrium problems with relaxed convexity requirements, in particular bifunctions which are strongly quasiconvex in their right argument. Here, a function $h:K\to\mathbb{R}$ is called strongly quasiconvex if
\[
h((1-\lambda) x+\lambda y)\leq \max\{h(x),h(y)\}-\lambda(1-\lambda)\frac{\gamma}{2}\norm{x-y}^2
\]
for all $x,y\in K$ and $\lambda\in [0,1]$, for a given value $\gamma>0$. This class of functions already goes back to the work of Polyak \cite{Polyak1966} and represents a restriction of the class of quasiconvex functions, that is the class of functions for which the above inequality holds with $\gamma = 0$. Nevertheless, it still accommodates various functions of practical interest, such as the Euclidean norm and its square root (see \cite{Jov1996,Lara2022}), and it in particular contains nonconvex functions (as e.g.\ witnessed by $\mathrm{max}\{\sqrt{\norm{\cdot}}, \norm{\cdot}^2-k\}$ for $k \in \N$, cf.\ \cite[Remark 18]{Lara2022}). We refer to \cite{GLM2025,Lara2022} for further discussions on the relationship between strongly quasiconvex functions and various other classes of (non)convex functions.

This class of strongly quasiconvex functions has various benefits over the class of plain quasiconvex functions mentioned above. In particular, many iterative procedures known from convex settings are still reasonably well-behaved (we in particular refer to \cite{Lara2022} for further discussions on this topic). This in particular is true for the proximal point algorithm, which was studied for these functions by Lara \cite{Lara2022}, who proved a convergence theorem over Euclidean space $\mathbb{R}^d$ (cf.\ \cite[Theorem 10]{Lara2022}). Contrary to the usual proximal point method for convex functions, the proximal map $\mathrm{Prox}_{\beta h}$ is however not single-valued anymore even over Euclidean spaces (cf.\ \cite[Remark 6]{Lara2022}), so that the proximal step is transformed into a selection $x_{k+1}\in \mathrm{Prox}_{\beta_k h}(x_k)$ (which is well-defined as $\mathrm{Prox}_{\beta h}$ is still non-empty). It is also in this setting that Iusem and Lara study the method sketched above in \cite{IusemLara2021}, that is they consider the iteration
\[
x_{k+1}\in \argmin_{y\in K}\left\{f(x_k,y) + \frac{1}{2\beta_k}\norm{y-x_k}^2\right\}:=\mathrm{Prox}_{\beta_kf(x_k,\cdot)}(x_k)
\]
for a (pseudomonotone and suitably Lipschitz) bifunction $f:K\times K\to\mathbb{R}$ on a closed and convex domain $K$ which is only strongly quasiconvex in its second argument, and establish its convergence over Euclidean space $\mathbb{R}^d$ (cf.\ \cite[Theorem 3.1]{IusemLara2021}). In particular, a key benefit of the above method to e.g.\ a method which considers a regularization of the bifunction as a whole in the style of the resolvent of a bifunction is that the above selection is immediately well-defined whereas the well-definedness of such resolvents of bifunctions is quite non-trivial in the context of strong quasiconvexity (see in particular the discussion in \cite{DespresPischke2026}).

The suitability of the above proximal point type method for equilibrium problems induced by strongly quasiconvex bifunctions is also further affirmed through the subsequent work of Grad, Lara and Marcavillaca \cite{GradLaraMarca24}, where they show that the method can be efficiently extended with both inertia terms and over-relaxations (extending previous work on minimization problems for strongly quasiconvex functions \cite{GradLaraMarcavillaca2023}), illustrating that more sophisticated formulations of proximal point type schemes are still viable in the context of the above class of nonconvex functions. Concretely, given a (pseudomonotone and suitably Lipschitz) bifunction $f:K \times K \to\mathbb{R}$ over an affine subspace $K\subseteq \mathbb{R}^d$ that is strongly quasiconvex in its right argument, the relaxed inertial proximal point algorithm considered in \cite{GradLaraMarca24} iteratively computes points $\{x_k\}_k$, $\{y_k\}_k$ and $\{z_k\}_k$ via
\[
\begin{cases}
y_k := x_k + \alpha_k(x_k - x_{k-1}),\\
z_{k} \in \mathrm{Prox}_{\beta_k f(y_k,\cdot)}(y_k),\\
x_{k+1} := (1-\rho_k)y_k + \rho_k z_k,
\end{cases}
\]
from given starting points $x_{-1}, x_0\in K$ and sequences $\{\alpha_{k}\}_k$, $\{\rho_k\}_k$, $\{\beta_k\}_k$ of parameters satisfying suitable conditions (discussed in detail later), and in \cite{GradLaraMarca24} the authors establish its convergence over $\mathbb{R}^d$ (cf.\ \cite[Theorem 11]{GradLaraMarca24}). 

While choosing $\alpha_k = 0$ and $\rho_k = 1$ for all $k \in \N$ allows one to recover the ``bare'' proximal point algorithm from \cite{IusemLara2021}, the strength of the above method lies of course in the use of non-trivial parameter sequences. In that context, the computation of $y_k$ in the above scheme, i.e.\ the so-called ``extrapolation step'', incorporates an inertia term as originally considered by Antipin \cite{Antipin1994}, Alvarez \cite{Alvarez2000} as well as Alvarez and Attouch \cite{AlvarezAttouch01} to improve the proximal point algorithm for monotone inclusions, motivated as a discretization of the ``heavy ball with friction'' system already studied by Polyak \cite{Polyak1966} (see also \cite{AttouchGoudouRedont2000}). The computation of $x_{k+1}$, i.e.\ the over-relaxation, follows a strategy first used by Eckstein and Bertsekas \cite{EcksteinBertsekas1992} for the proximal point algorithm for monotone inclusions and has the particular benefit that it may indeed speed up the convergence of the method (see e.g.\ \cite{Bertsekas1982,EcksteinFerris1998}), in particular in combination with inertia terms which was first considered by Alvarez \cite{Alvarez04}. For convex equilibrium problems, such a relaxed-inertia proximal point type method was, to our knowledge, first considered in \cite{HieuDuongThai21} (see also the related second order dynamical system studied in \cite{VanVinTran22}).

\subsection{The contributions of the present paper and related work}

In the present paper, we study the relaxed inertial proximal point method as presented above for suitable strongly quasiconvex equilibrium problems over Hadamard manifolds, that is complete simply connected Riemannian manifolds of nonpositive sectional curvature. In many ways, Hadamard manifolds provide a very suitable but nontrivial framework for convex analysis and optimization in nonlinear contexts, and we refer to \cite{Bacak2014,Bacak23} and the references therein for discussions in that direction.

In particular, the proximal point algorithm for convex minimization and its extension for inclusion problems for monotone vector fields have been studied deeply over Hadamard manifolds. We refer in particular \cite{FerreiraOliveira2002,LiLopezMartinMarquez2009,PapaQuirozOliveira2009} (see also \cite{Bac2013}), next to many others. Extensions of this method for convex equilibrium problems are similarly quite well-studied and can in particular be found in \cite{BentoCruzNetoLopesMeloFilho2024,BentoCruzNetoMelo2022,COLAO201261} (see also \cite{KHATIBZADEHMOHEBBI2021}), all of which however use the resolvent of a bifunction in various guises, that is they approach the problem by regularizing the bifunction as a whole. For quasiconvex functions, the literature is largely limited to the setting of minimization problems, as e.g.\ studied in \cite{Quiroz24,QuirzoAlexCusiMac20,PapaQuirozOliveira2009,QO2012a,QO2012b}. In particular, work on problems induced by strongly quasiconvex functions in these contexts is very limited so far and, to our knowledge, only comprises the recent works \cite{Pischke2025} of the second author and \cite{DespresPischke2026} of the authors which, respectively, extend the convergence results for the proximal point type method for strongly quasiconvex minimization \cite{Lara2022} and equilibrium problems \cite{IusemLara2021} obtained in the linear setting to Hadamard spaces, that is complete geodesic metric spaces of nonpositive curvature (see  \cite{Bacak2014,Bacak23,BridsonHaefliger1999}), a class that in particular includes Hadamard manifolds but is neither restricted to finite-dimensionality nor smoothness assumptions.

The present work can thus be seen as an extension of the results obtained in \cite{DespresPischke2026,Pischke2025} to the setting of \cite{GradLaraMarca24}. In particular, we rely on the general approach initiated in \cite{DespresPischke2026,Pischke2025} to also study the present manifold variant of the method from \cite{GradLaraMarca24}, which concretely takes the following form: Over a Hadamard manifold $X$ with Riemannian metric $d$, we iteratively compute points $\{x_k\}_k$, $\{y_k\}_k$ and $\{z_k\}_k$ via
\[
\begin{cases}
y_k := \exp_{x_k}(-\alpha_k\exp^{-1}_{x_k}x_{k-1}),\\
z_{k} \in \argmin_{x\in K}\left\{f(y_k,x)+ \frac{1}{2\beta_k}d^2(y_k,x)\right\},\\
x_{k+1} := \exp_{z_k}((1-\rho_k)\exp^{-1}_{z_k}y_k),
\end{cases}
\]
from given starting points $x_{-1}, x_0\in K$ and sequences $\{\alpha_{k}\}_k$, $\{\rho_k\}_k$, $\{\beta_k\}_k$ of parameters as before (see Section \ref{sec:alg} for details). Here, $\exp$ is the exponential map of the underlying manifold (see Section \ref{sec:HadamardMani} for details) and, correspondingly, $K\subseteq X$ is now a metrically closed set that contains all geodesic lines determined by any two points of $K$, representing a suitable analog for affine subspaces in Hadamard manifolds (see Section \ref{sec:alg} for details).

Again, for the special choices of $\alpha_k = 0$ and $\rho_k = 1$ for all $k \in \N$, the above method reduces to the ``bare'' proximal point algorithm previously studied by the authors in \cite{DespresPischke2026} over Hadamard spaces (extending \cite{IusemLara2021}). The use of the exponential map to formulate inertia terms via
\[
\exp_{x_k}(-\alpha_k\exp^{-1}_{x_k}x_{k-1})
\]
as above in the context of Hadamard manifolds goes back at least to \cite{KhammahawongChaipunyaKumam2023} (see also \cite{ChangYaoLiuZhao2023,SahuSharmaGautam2026}), and we crucially rely on the approach taken therein to handle the above extrapolation step in this paper. While one often finds relaxations, that is steps of the form
\[
\exp_{z_k}((1-\rho_k)\exp^{-1}_{z_k}y_k)
\]
as above for the choice of $\rho_k\in [0,1]$, in the literature on optimization methods in Hadamard manifolds (as in e.g.\ \cite{KhammahawongChaipunyaKumam2023,SahuSharmaGautam2026}, among many others), this does not seem to be true for over-relaxations, that is for the choice of $\rho_k\in [0,2]$, where the present paper seems to provide the first such treatment, to our knowledge. Perhaps unsurprisingly, we for this rely on similar machinery as needed to handle inertia terms in this nonlinear context.

Our main result is then a convergence theorem for this method over Hadamard manifolds (see Theorems \ref{THEOREM3} and \ref{THEOREM3a}) under suitable conditions on the parameters and the bifunction (derived from \cite{GradLaraMarca24}, as discussed in detail in Section \ref{sec:alg} later). In particular, this extends the convergence result obtained in \cite{GradLaraMarca24} for the first time to a nonlinear setting. Modeled after the approach of \cite{DespresPischke2026,Pischke2025}, our convergence proofs are quite elementary and in particular effective, contrary to \cite{GradLaraMarca24}, avoiding any kind of infinitary arguments like (weak) compactness. This was not only instrumental for obtaining our results in a nonlinear context but also allows for the formulation of explicit and fast rates of convergence (similar to \cite{DespresPischke2026,Pischke2025}) for the sequences $\{x_k\}$, $\{y_k\}$ and $\{z_k\}$ generated by the above method. More precisely, we in particular establish the following quantitative result: 
\begin{equation*}
\forall \varepsilon > 0\ \forall k \geq \left(\left \lceil \frac{C}{\varepsilon^2} \right \rceil+ M(E \varepsilon^2) + 2\right) \left(d(x_k, x^{*}), d(y_k, x^{*}), d(z_k, x^{*})< \varepsilon\right),
\end{equation*}
where $x^*$ is the (unique) solution to the associated equilibrium problem. The corresponding rate $\left(\left \lceil C/\varepsilon^2 \right \rceil+ M(E \varepsilon^2) + 2\right)$ is in particular very uniform,  depending only on a rate of convergence $M$ for the sum $\sum_{k=0}^\infty \alpha_k d^2(x_k,x_{k-1})<+\infty$ (see Section \ref{sec:convergence} later on) as well as two constants $C$, $E$ which can be explicitly calculated  only in terms of a few surrounding numerical parameters of the problem. To our knowledge, these are the first effective results for this method, even in the original setting of Euclidean spaces from \cite{GradLaraMarca24}. Similar to \cite{DespresPischke2026}, in the course of our convergence proofs, we also show how some semicontinuity assumptions on the bifunction featuring in \cite{GradLaraMarca24} can be avoided (see Remark \ref{rem:generalization} later on).

While we have chosen to formulate our results over Hadamard manifolds for simplicity, it should be noted that the present approach does not at all depend on the smoothness of the space $X$ or any other differential structure and could in principle be carried out, mutatis mutandis, over a suitable class Hadamard spaces that is discussed in detail in Remark \ref{rem:Hilbert} later on. In particular, this class has no dimensionality restrictions in general and in particular contains Hilbert spaces and some Hilbert manifolds. In that context, the results of the present paper in particular show that (an appropriate formulation of) the method above strongly converges even in these infinite-dimensional spaces, similar to \cite{DespresPischke2026,Pischke2025}. Further, the quantitative estimates remain exactly the same. A detailed discussion of the Hilbert space case can be found in the master thesis of the first author \cite{Despres2026}, on which this paper is based.

Lastly, while we focus only on equilibrium problems in this paper, our approach should immediately apply to minimization problems, not only trivially by instantiating the problem and method above, but also in the context of the generalized parameter assumptions used in that context in \cite{GradLaraMarcavillaca2023}. We however do not spell this out here.

\subsection{Proof mining}

The present quantitative results have been obtained using methods from proof mining, a program in mathematical logic that utilizes tools from proof theory, the foundations of mathematics, to classify and extract the computational content of prima facie ``non-computational'' proofs from various areas of core mathematics. For a comprehensive overview of this area and some of the many applications to nonlinear analysis, we refer to the seminal monograph \cite{Kohlenbach2008} as well as to the survey \cite{Kohlenbach18Survey}. As typical for many applied results of proof mining, the present paper is presented without any explicit reference to logic and does not require any such background. We instead refer to the master thesis of the first author \cite{Despres2026} for a detailed discussion on the concrete use of logical methods in the context of the present paper.

\subsection{Outline of the present paper}

The paper is now organized as follows: In Section \ref{sec:prelim}, we give the necessary background on Hadamard manifolds and the related objects as well as notions of convexity and quasiconvexity, together with some basic lemmas needed for the main results. Section \ref{sec:alg} carefully introduces the method sketched above in Hadamard manifolds and discusses the central assumptions on the surrounding objects. Finally, Section \ref{sec:convergence} then provides the corresponding (quantitative) convergence theorems.

\section{Preliminaries}\label{sec:prelim}

\subsection{Hadamard manifolds}\label{sec:HadamardMani}

We now first give the (very few) required preliminaries for Riemannian and Hadamard manifolds. For further exposition beyond the results given here, we refer to \cite{Sakai1996} (as well as \cite{Bacak2014,BridsonHaefliger1999}).

In the following, and for the rest of the paper unless stated otherwise, we take $X$ to be a finite dimensional Hadamard manifold, that is a complete simply connected Riemannian manifold of nonpositive sectional curvature, with a Riemannian metric $\langle\cdot,\cdot\rangle$. We denote its tangent space at a point $x\in X$ by $T_xX$ and we denote the norm on $T_xX$ induced by the Riemmanian metric by $\norm{\cdot}$. The Riemannian distance on $X$, defined as usual, is denoted by $d$. It induces the original topology on $X$ and, in particular, $(X,d)$ is a complete metric space. 

Given an interval $I\subseteq\mathbb{R}$, a geodesic in $X$ is a map $\gamma:I\to X$ such that
\[
d(\gamma(t),\gamma(s))=c\vert t-s\vert
\]
for all $t,s\in I$, where $c$ is called the speed of $\gamma$. If $I=[0,1]$, we say that $\gamma$ is a normalized geodesic and we say that $\gamma$ joins $x=\gamma(0)$ and $y=\gamma(1)$. Note that in this case $c=d(x,y)$. If $I=[0,\infty)$, we call $\gamma$ a geodesic ray and if $I=\mathbb{R}$, we call $\gamma$ a geodesic line. By the Cartan-Hadamard theorem (cf.\ \cite[Theorem 4.1]{Sakai1996}), for any two points $x,y\in X$, there exists a unique normalized geodesic joining $x$ to $y$. We call a subset $K \subseteq X$ convex if with $x,y \in K$ also $\gamma(t)\in K$ for all $t\in [0,1]$ where $\gamma$ is the unique normalized geodesic joining $x$ to $y$.

Given a point $x\in X$ and vector $\xi\in T_xX$, the exponential map is defined by $\exp_x \xi=\gamma_{x,\xi}(1)$ where $\gamma_{x,\xi}$ is the unique normalized geodesic with $\gamma_{x,\xi}(0)=x$ and velocity $\dot{\gamma}_{x,\xi}(0)=\xi$. It can be shown that $\exp_x(t\xi)=\gamma_{x,\xi}(t)$ for all $t\in [0,1]$. The map $\exp_x:T_xX\to X$ is a diffeomorphism and in particular has an inverse $\exp^{-1}_x:X\to T_xX$ where it holds that $d(x,y)=\norm{\exp^{-1}_xy}$ for all $x,y\in X$. It follows that if $x\neq y$, then $r(t)=\exp_x (t\exp^{-1}_xy)$ is a geodesic line in $X$ with $r(0)=x$ and $r(1)=y$.

Metrically, $(X,d)$ is a Hadamard space, that is a complete geodesic metric space of nonpositive curvature, also known as a (complete) $\mathrm{CAT}(0)$ space (see in particular \cite{Bacak2014,BridsonHaefliger1999}). As such, its so-called quasi-inner product as introduced by Berg and Nikolaev \cite{BergNikolaev2008}, that is the function
\[
\langle\vv{xy},\vv{uv}\rangle:= \frac{1}{2}\big(d^2(x,v)+d^2(y,u)-d^2(x,u)-d^2(y,v)\big),
\]
satisfies the following analogue (cf.\ \cite[Corollary 3]{BergNikolaev2008}) of the Cauchy-Schwarz inequality 
\begin{equation}
\langle\vv{xy},\vv{uv}\rangle \leq d(x,y)d(u,v)\tag{CS}\label{CS}
\end{equation}
for all $x,y,u,v \in X$, next to the following four further (in particular characterizing, see \cite[Proposition 14]{BergNikolaev2008}) properties for any $x,v,u,v\in X$: $\langle \vv{xy},\vv{xy}\rangle = d^2(x,y)$; $\langle \vv{xy},\vv{uv}\rangle = \langle \vv{uv},\vv{xy}\rangle $;  $\langle \vv{xy},\vv{uv}\rangle = -\langle \vv{yx},\vv{uv}\rangle $; $\langle \vv{xy},\vv{uv}\rangle + \langle \vv{xy},\vv{vw}\rangle = \langle \vv{xy},\vv{uw}\rangle$. 

Beyond the above, we rely on two key inequalities for the Riemannian distance. The first is an extension of the so-called Bruhat-Tits $\mathrm{CN}$-inequality \cite{BruhatTits1972} to geodesics, which in particular characterizes Hadamard spaces:

\begin{lemma}[{cf.\ \cite[Theorem 1.3.3]{Bacak2014}}]\label{CN}
Let $x,y,w\in X$ and, given $t\in [0,1]$, define $z:=\exp_{x} (t\exp^{-1}_{x}y)$. Then
\[
d^2(z,w) \leq  (1-t)d^2(x,w)+t d^2(y,w)-t(1-t)d^2(x,y).
\]
\end{lemma}

The next is a kind of dual to the former, formulated for negative values. In particular, it will be key for our handling of the inertia terms as well as the over-relaxations.

\begin{lemma}[{cf.\ \cite[Proposition 9]{SahuSharmaGautam2026}}]\label{revCN}
Let $x,y,w\in X$ and, given $t\in [0,1]$, define $z:=\exp_{x} (-t\exp^{-1}_{x}y)$. Then
\[
d^2(z,w) \leq  (1+t)d^2(x,w)-t d^2(y,w)+t(1+t)d^2(x,y).
\]
\end{lemma}

While Proposition 9 in \cite{SahuSharmaGautam2026} states the above as a general and abstract result, it should be noted that it already implicitly appears in the proof of Theorem 1 in \cite{KhammahawongChaipunyaKumam2023}.

\subsection{Quasiconvexity over Hadamard manifolds}

Over a Hadamard manifold $X$, a function $h: X \rightarrow \R$ is called \emph{convex} if 
\begin{equation*}
(h\circ\gamma)(\lambda)\leq (1-\lambda) h(\gamma(0)) + \lambda h(\gamma(1))
\end{equation*}
for any geodesic $\gamma:[0,1]\to X$ and any $\lambda \in [0,1]$. Going beyond convex functions, the present paper will focus on the weaker class of quasiconvex functions and their so-called strongly quasiconvex variants. In that context, a function $h: X \rightarrow \R$ is called
\begin{itemize}
\item[(i)]\emph{quasiconvex} if 
\begin{equation*}
(h\circ\gamma)(\lambda) \leq \mathrm{max}\{h(\gamma(0)),h(\gamma(1))\}
\end{equation*}
for any geodesic $\gamma:[0,1]\to X$ and any $\lambda \in [0,1]$,
\item[(ii)]\emph{strongly quasiconvex} with modulus $\gamma > 0$ if 
\begin{equation*}
(h\circ\gamma)(\lambda) \leq \mathrm{max}\{h(\gamma(0)),h(\gamma(1))\}- \lambda(1-\lambda)\frac{\gamma}{2}d^2(\gamma(0),\gamma(1))
\end{equation*}
for any geodesic $\gamma:[0,1]\to X$ and any $\lambda \in [0,1]$.
\end{itemize}

Next to notions of convexity, we will also need the following notions of semicontinuity: A function $h:X \rightarrow \mathbb{R}$ is called \emph{lower semicontinuous} (lsc) at $x \in X$ if 
\begin{equation*}
h(x) \leq \liminf_{k \rightarrow \infty} h(x_k)
\end{equation*}
for any sequence $\{x_k\}_k \subseteq X$ with $x_k \rightarrow x$. If instead it holds that 
\begin{equation*}
h(x) \geq \limsup_{k \rightarrow \infty} h(x_k)
\end{equation*}
for any sequence $\{x_k\}_k \subseteq X$ with $x_k \rightarrow x$, $h$ is called \emph{upper semicontinuous} (usc) at $x \in X$.

Given any function $h: X \rightarrow \R$, we define its proximal map $\mathrm{Prox}_{\beta h}$ by 
\begin{equation*}
\mathrm{Prox}_{\beta h}(x) := \argmin_{y\in X}\left\{h(y) + \frac{1}{2\beta}d^2(x,y)\right\}
\end{equation*}
for $\beta>0$ and $x \in X$. As remarked in the introduction, $\mathrm{Prox}_{\beta h}$ is no longer a singleton for strongly quasiconvex functions but is non-empty under suitable assumptions (see also Remark \ref{rem:wellDef} later). The main result for the proximal map that we utilize beyond that is the following:

\begin{lemma}[{cf.\ \cite[Lemma 3.9]{Pischke2025}}]\label{Lemma11P}
Let $h: X \rightarrow \R$ be a strongly quasiconvex function with modulus $\gamma >0$ and let $\beta > 0$ and $x\in X$. If $\overline{x}\in \mathrm{Prox}_{\beta h}(x)$, then
\[
h(\overline{x}) \leq \max\{h(y),h(\overline{x})\} + \frac{\lambda}{2}\Big(\frac{\lambda}{\beta}-\gamma + \lambda \gamma \Big )d^2(y,\overline{x}) + \frac{\lambda}{\beta}\langle \vv{x\overline{x}}, \vv{\overline{x}y} \rangle
\] 
for all $y \in X$ and $\lambda \in [0,1]$.
\end{lemma}

The above result, which extends \cite[Proposition 7]{Lara2022}, was derived in \cite{Pischke2025} already over Hadamard spaces, and hence holds over Hadamard manifolds as well.

\section{Relaxed-inertial proximal point algorithms on Hadamard manifolds}\label{sec:alg}

We now begin by defining the main method studied in this paper for solving the equilibrium problem \eqref{equil} over a Hadamard manifold $X$, a suitable subspace $K\subseteq X$ and a given bifunction $f: K\times K \rightarrow \R$. We use $S(K,f)$ to denote the associated solution set. As discussed in the introduction, we base our method on that introduced and studied in the work of Grad, Lara and Marcavillaca \cite{GradLaraMarca24} for the associated problem in Euclidean spaces, and so first recall that method and its associated parameters here in detail.

Over $X=\mathbb{R}^d$, let $K\subseteq X$ be an affine subspace. Fix $\alpha, \rho \in [0,1)$ and starting points $x_0, x_{-1} \in K$ as well as sequences of parameters $\{ \beta_{k}\}_k \subseteq (0,\infty)$, $\{ \alpha_{k}\}_k \subseteq [0,\alpha]$ and $\{\rho_k\}_k \subseteq [1-\rho,1+\rho]$. Iteratively, the method introduced in \cite{GradLaraMarca24} now calculates a new point $x_{k+1}\in K$ from previous points $x_{k},x_{k-1}\in K$ as follows:
\[
\begin{cases}
y_k := x_k + \alpha_k(x_k - x_{k-1}),\\
z_{k} \in \argmin_{x\in K}\left\{f(y_k,x)+ \frac{1}{2\beta_k}\norm{y_k-x}^2\right\},\\
x_{k+1} := (1-\rho_k)y_k + \rho_k z_k.
\end{cases}\tag{\%}\label{preRIPPA}
\]

The first step $y_k$, commonly called the exploration step, incorporates the so-called inertia term $\alpha_k(x_k - x_{k-1})$ already discussed in the introduction. The second step $z_k$, commonly called the proximal step, then produces a new point from this $y_k$ by utilizing a proximal regularization of the bifunction in its second argument, with the first argument fixed to $y_k$, akin to the previous proximal method studied by Iusem and Lara \cite{IusemLara2021} over Euclidean spaces (and the present authors over Hadamard spaces \cite{DespresPischke2026}). The third step $x_{k+1}$, commonly called the (over)relaxation step, combines the previous two.

In \cite{GradLaraMarca24}, Grad, Lara and Marcavillaca study the above algorithm for a specific class of non-convex bifunctions, characterized by the following conditions $(A1)$ -- $(A5)$:

\begin{enumerate}
\item [$(A1)$] For every $x \in K$, the function $f(x,\cdot)$ is lsc, and for every $y \in K$, the function $f(\cdot,y)$ is usc.
\item [$(A2)$] $f$ is pseudomonotone, i.e.
\[
f(x,y) \geq 0 \text{ implies } f(y,x) \leq 0
\]
for all $x,y \in K$.
\item [$(A3)$] $f$ is lsc (jointly in both arguments).
\item [$(A4)$] For every $x \in K$, the function $f(x,\cdot)$ is strongly quasiconvex with modulus $\gamma > 0$.
\item [$(A5)$] $f$ satisfies the following Lipschitz condition: there exists $\eta > 0$ such that
\[
f(x,z) - f(x,y) - f(y,z) \leq \eta \left(\norm{x-y}^2 + \norm{y-z}^2 \right)
\]
for all $x,y,z \in K$.
\end{enumerate}

Further, they rely on the following crucial assumptions on the associated parameter sequences: For $\{\beta_k\}_k$, it is assumed that there exists an $e>0$ such that 
\begin{enumerate}
\item[$(C1)$] $\frac{1}{\gamma-8\eta} < \beta_k < e \leq \upperbound$ for all $k \in \N$,
\end{enumerate}
and, relative to this $e$, they further require that 
\begin{enumerate}
\item[$(C2)$] $0 < 1-\rho \leq \rho_k \leq 1+\rho$ with $0 \leq \rho \leq 1- 4\eta e$.
\end{enumerate}

In both cases, $\gamma$ and $\eta$ are of course as in $(A4)$ and $(A5)$ respectively. Similar to \cite[Remark 8]{GradLaraMarca24}, we remark that if one considers $e = \upperbound$, $(C2)$ implies $\rho = 0$ and so $\rho_k = 1$ for all $k \in \N$. In this case, $(C1)$ becomes $\lowerbound < \beta_k < \upperbound$ for all $k \in \N$ as employed in the context of the proximal point method studied in \cite{IusemLara2021} (as well as \cite{DespresPischke2026}). In any case, $(C1)$ in this way in particular implies that $12 \eta < \gamma$, which featured as assumption $(A6)$ in \cite{IusemLara2021} (as well as \cite{DespresPischke2026}), and we will likewise often use this property in the present paper. Further as in \cite{IusemLara2021}, we note that the standard assumption
\begin{enumerate}
\item [$(A0)$] $f(x,x)=0$ for all $x\in X$
\end{enumerate}
can be derived from $(A2)$ and $(A5)$, but we will later also use it individually.

In the setting of $(A1)$ -- $(A5)$ and $(C1)$ -- $(C2)$, Grad, Lara and Marcavillaca \cite{GradLaraMarca24} show that the above method \eqref{preRIPPA} is well-defined and that it converges to (the unique) solution of the associated equilibrium problem under a suitable decaying assumption on the inertia terms:

\begin{theorem}[{cf.\ \cite[Theorem 11]{GradLaraMarca24}}]\label{GLMconv}
Let $X=\mathbb{R}^d$ and let $K\subseteq X$ be an affine subspace. Suppose that $f:K\times K\to\mathbb{R}$ satisfies $(Ai)$ for $i = 1,2,3,4,5$. Fix $\alpha, \rho \in [0,1)$ and $x_0, x_{-1} \in K$ as well as $\{ \beta_{k}\}_k \subseteq (0,\infty)$, $\{ \alpha_{k}\}_k \subseteq [0,\alpha]$ and $\{\rho_k\}_k \subseteq [1-\rho,1+\rho]$ such that $(C1)$ and $(C2)$ hold, and let $\{x_k\}_k, \{y_k\}_k$ and $\{z_k\}_k$ be the sequences generated by \eqref{preRIPPA}. If
\[
\sum_{k=0}^{\infty} \alpha_k \norm{x_k-x_{k-1}}^2 < +\infty,
\]
then the sequences $\{x_k\}_k, \{y_k\}_k$ and $\{z_k\}_k$ converge to the same limit $x^*$, where $\{x^{*}\} = S(K,f)$.
\end{theorem}

\begin{remark}
In \cite{GradLaraMarca24}, assumption $(A1)$ only refers to the requirement that $f(\cdot,y)$ is usc for all $y \in K$. The assumption that $f(x,\cdot)$ is lsc for every $x \in K$ is denoted by $(A1')$ in \cite{GradLaraMarca24} and is treated separately as it naturally follows from $(A3)$. However, as we will in particular show in this paper that the assumption $(A3)$ can in certain contexts be eliminated from the above convergence theorem and its generalization to Hadamard manifolds (see in particular Remark \ref{rem:generalization} later on), we prefer the above naming convention, which follows that of Iusem and Lara \cite{IusemLara2021} (as well as \cite{DespresPischke2026}). Further, we want to note that a slight weakening of the assumptions on the parameters is discussed in \cite[Remark 87]{GLM2025} and while we deem it quite likely that our present results can be extended to this weakened assumption, we here focus on the original assumptions from \cite{GradLaraMarca24} for simplicity.
\end{remark}

Lifted to a general Hadamard manifold $X$, the above method \eqref{preRIPPA} now takes the following form: First, to set up the problem, we fix a set $K\subseteq X$ which is metrically closed and such that the complete geodesic line determined by any two points of $K$ lies in $K$, that is
\[
\exp_x(t\exp^{-1}_xy)\in K\text{ for all }t\in\mathbb{R}\text{ and }x,y\in K.
\]
In particular, $K$ is a convex set and so $K$ is a Hadamard space (see e.g.\ \cite{Bacak2014}). Over linear spaces $X$, sets $K$ as above are exactly the affine subspaces and we will find that the above closure conditions are the right conditions to generalize the method from \cite{GradLaraMarca24}.

Over such a $K$, we then consider a bifunction $f:K\times K\to\mathbb{R}$ satisfying assumptions among $(A1)$ -- $(A5)$, with $(A5)$ now formulated using the metric via
\[
f(x,z) - f(x,y) - f(y,z) \leq \eta \left(d^2(x,y) + d^2(y,z) \right)
\]
for all $x,y,z \in K$, using a fixed $\eta > 0$. We still denote the set of equilibrium points of $f$ over $K$, that is the set of all $x^{*} \in K$ such that $f(x^{*},y) \geq 0$ for all $y \in K$, by $S(K,f)$.  As $K$ is a Hadamard space, \cite[Theorem 3.9]{DespresPischke2026} guarantees that $S(K,f)\neq\emptyset$ whenever $f$ satisfies $(A0)$, $(A1)$, $(A2)$ and $(A4)$.

We now move on to the method. Fixing parameters $\alpha, \rho \in [0,1)$ and $x_0, x_{-1} \in K$ as well as $\{ \beta_{k}\}_k \subseteq (0,\infty)$, $\{ \alpha_{k}\}_k \subseteq [0,\alpha]$ and $\{\rho_k\}_k \subseteq [1-\rho,1+\rho]$ as above, the method iteratively calculates a new point $x_{k+1}\in K$ from the previous points $x_{k},x_{k-1}\in K$ as follows:
\[
\begin{cases}
y_k := \exp_{x_k}(-\alpha_k\exp^{-1}_{x_k}x_{k-1}),\\
z_{k} \in \argmin_{x\in K}\left\{f(y_k,x)+ \frac{1}{2\beta_k}d^2(y_k,x)\right\},\\
x_{k+1} := \exp_{z_k}((1-\rho_k)\exp^{-1}_{z_k}y_k).\tag{RIPPA}\label{RIPPA}
\end{cases}
\]

Again, as outlined in the introduction, the first step $y_k$ represents the exploration step involving the inertia term in a manifold. The second step represents the proximal step as usual and the last represents an (over)relaxation in manifolds, combining the previous two. Using the fact that $\exp_{x}(\xi)=x+\xi$ and $\exp^{-1}_x(y)=y-x$ over Euclidean spaces, it can be easily verified that the method \eqref{RIPPA} coincides with \eqref{preRIPPA} in such a setting.

The first observation we want to make is that, like \eqref{preRIPPA} in \cite{GradLaraMarca24}, also the above \eqref{RIPPA} is well-defined in the context of the parameter restrictions given by $(C1)$ and $(C2)$. This crucially relies on the fact that
\[
\mathrm{Prox}_{\beta_k f(y_k,\cdot)}:=\argmin_{x\in K}\left\{f(y_k,x)+ \frac{1}{2\beta_k}d^2(y_k,x)\right\}\neq\emptyset,
\]
which follows rather immediately from corresponding results obtained by the authors in \cite{DespresPischke2026} over Hadamard spaces. We collect this in the following remark.

\begin{remark}\label{rem:wellDef}
By \cite[Proposition 4.4]{DespresPischke2026}, given a bifunction $f:Y\times Y\to\mathbb{R}$ on a Hadamard space $Y$ which is quasiconvex and lsc in its right argument and satisfies $(A5)$, the proximal map $\mathrm{Prox}_{\beta f(z,\cdot)}$ is non-empty for any $z\in X$ and $\beta\leq\frac{1}{4\eta}$ with $\eta$ as in $(A5)$. As $K$ is in particular a Hadamard space, this result also applies here, so that \eqref{RIPPA} is well-defined whenever $(A1)$, $(A4)$, $(A5)$ and $(C1)$ hold: For $x_k,x_{k-1}\in K$, we have $y_k\in K$ since $K$ is closed under complete geodesic lines. The above result yields that $z_k\in K$ is well-defined. We get $x_{k+1}\in K$, again using that $K$ is closed under complete geodesic lines.
\end{remark}

\begin{remark}
As utilized in \cite{GradLaraMarca24}, the convenient stopping criterion $y_{k} = z_k$ applies in the context \eqref{preRIPPA} as whenever this equality holds, one already knows that $\{y_k\} = S(K,f)$. While \ref{RIPPA} and \eqref{preRIPPA} above were formulated without such a condition (contrary to \cite{GradLaraMarca24}) to focus purely on the (infinite) sequences generated by these methods, we want to highlight that a similar stopping criterion also applies in the nonlinear context of \eqref{RIPPA}. Indeed, under $(A0)$ and $(A4)$ it follows from \cite[Proposition 4.6]{DespresPischke2026} that 
\[
S(K,f) = \mathrm{Fix}\left(\argmin_{y \in K}\left\{f(\cdot, y) + \frac{1}{2\beta} d^2(\cdot,y)\right\}\right)
\]
for any $\beta > 0$, again using that $K$ is a Hadamard space.
\end{remark}

\section{Quantitative convergence theorems}\label{sec:convergence}

\subsection{Key assumptions}

We now move to the main results of the present paper, that is quantitative convergence theorems for the method \eqref{RIPPA} introduced above over general Hadamard manifolds and for suitable bifunctions. In order to derive our effective rates of convergence, we will need to consider one of two distinct \emph{quantitative} variants of the condition $(C1)$. 

Concretely, one of the two conditions we consider on the one hand strengthens the property $\beta_k < e$ to $\beta_k \leq e-\Theta$ for some $\Theta>0$ with $e\leq\upperbound$. However, the condition we consider simultaneously weakens the lower bound from $\frac{1}{\gamma-8\eta} < \beta_k$ to $\frac{1}{\gamma-8\eta} \leq \beta_k$. Together, we arrive at the following:
\begin{enumerate}
\item [${(C1)}^l$] $\frac{1}{\gamma-8\eta} \leq \beta_k \leq e-\Theta$ for all $k \in \N$, where $e\leq \upperbound$ and $\Theta>0$ is some (fixed) constant.
\end{enumerate}

The dual to this arises by instead making the lower bound $\frac{1}{\gamma-8\eta} < \beta_k$ uniformly strict by the introduction of a corresponding additional constant $\theta$, while the upper bound $\beta_k < e \leq \upperbound$ is relaxed to a non-strict variant, that is:

\begin{enumerate}
\item [${(C1)}^u$] $\lowerbound + \theta \leq \beta_k \leq e$ for all $k \in \N$, where $e\leq \upperbound$ and $\theta>0$ is some (fixed) constant.
\end{enumerate}

Recall that, as discussed before in the context of assumption $(C1)$, naturally also these extended assumptions ${(C1)}^l$ or ${(C1)}^u$ imply $12 \eta < \gamma$, which we utilize repeatedly throughout. Also recall that, as mentioned in Remark \ref{rem:wellDef}, the iteration \eqref{RIPPA} is well-defined in the context of $(A1)$, $(A4)$, $(A5)$ and $(C1)$ but by the argument given therein, the same holds for $(C1)$ replaced by ${(C1)}^l$ or ${(C1)}^u$.

Our quantitative results given later depend on assuming either ${(C1)}^l$ or ${(C1)}^u$. We are not able to give similar quantitative results for the original condition $(C1)$, nor are we able to relax any of the two assumptions ${(C1)}^l$ or ${(C1)}^u$ any further, e.g.\ to 
\[
\frac{1}{\gamma-8\eta} \leq \beta_k \leq e \leq \upperbound \text{ for all }k\in\mathbb{N}.
\]
This should in particular be compared to \cite{DespresPischke2026}, where in the context of similar quantitative results and extensions to nonlinear spaces for the proximal point type method \emph{without} inertia terms or over-relaxations featuring in \cite{IusemLara2021} (which arises from \eqref{RIPPA} by setting $\alpha,\rho=0$), it was possible to weaken the condition $\{\beta_k\}_k\subseteq (\lowerbound,\upperbound)$ to $\{\beta_k\}_k\subseteq [\lowerbound,\upperbound]$. In fact, the approach taken later for deriving our (quantitative) convergence result under the condition ${(C1)}^u$ relies on the machinery developed in \cite{DespresPischke2026} (and further extensions thereof given in this paper, see in particular Lemma \ref{Nporps} later on) for treating this generalization.

Practically, this however seems to be of rather little importance as $\{\beta_k\}_k$ is an auxiliary sequence which can be chosen freely, so that the above ${(C1)}^l$ or ${(C1)}^u$ are in a way as easy to fulfill as $(C1)$ itself.

Similar to \cite{DespresPischke2026}, in the course of our convergence analysis, we will assume the well-definedness of \eqref{RIPPA} outright, and this, paired with the assumption of the existence of a solution to \eqref{equil}, which is guaranteed under $(A0)$, $(A1)$, $(A2)$ and $(A4)$, will allow us to discharge any continuity assumptions like $(A1)$ or $(A3)$ (see Remark \ref{rem:generalization} later on). In particular, under the assumption of either ${(C1)}^l$ or ${(C1)}^u$, $(A3)$ can be completely removed. In that way, unless stated otherwise, we from now on fix $\alpha, \rho \in [0,1)$ and $x_0, x_{-1} \in K$ as well as $\{ \beta_{k}\}_k \subseteq (0,\infty)$, $\{ \alpha_{k}\}_k \subseteq [0,\alpha]$ and $\{\rho_k\}_k \subseteq [1-\rho,1+\rho]$. Further, we suppose that \eqref{RIPPA} is well-defined and let $\{x_k\}_k, \{y_k\}_k$ and $\{z_k\}_k$ be its generated sequences. Further, we assume that $S(K,f) \neq \emptyset$ and denote its unique element by $x^{*}$.

\subsection{Basic lemmas}

In similarity to \cite{DespresPischke2026} (and by relation also \cite{GradLaraMarca24,IusemLara2021}, although this is left rather implicit therein), a key step for our convergence analysis involves establishing the \emph{Fej\'er monotonicity} of the method. However, contrary to \cite{DespresPischke2026,IusemLara2021}, we here rely on establishing the \emph{quasi-Fej\'er monotonicity} of the method \eqref{RIPPA} (already featuring implicitly in \cite{GradLaraMarca24} for the method \eqref{preRIPPA}), that is we show that 
\begin{equation*}
\forall k,l\in \N \left( d^2(x_{k+l},x^{*}) \leq a \cdot d^2(x_k, x^{*}) + \sum_{i = k }^{\infty} e_i \right)
\end{equation*}
where $a >0$ and $\{e_i\}_i$ is a sequence in $(0, \infty)$ with $\sum_{i = 0}^{\infty} e_i < +\infty$.

Towards this property, we rely on the abstract Lemma \ref{madness0}, which is a slightly corrected version (see Remark \ref{GradLaraMarcaWrong} below) of the technical \cite[Lemma 2, (a)]{GradLaraMarca24}, which itself was essentially established already in \cite[Theorem 2.1]{AlvarezAttouch01}. As the proof just amounts to some simple (but slightly tedious) calculations and inductions, we omit it.

\begin{lemma}[{essentially \cite[Lemma 2, (a)]{GradLaraMarca24}}]\label{madness0}
Let $\{\varphi_k\}_k, \{s_k\}_k,\{\delta_k\}_k\subseteq [0,\infty)$, let $\varphi_{-1} \in [0,\infty)$ and let $\alpha \in [0,1)$ with $\{ \alpha_{k}\}_k \subseteq [0,\alpha]$. Assume that
\begin{equation}
\varphi_{k+1} - \varphi_k + s_{k+1} \leq \alpha_k(\varphi_k - \varphi_{k-1}) + \delta_k \text{ for all } k \in \N.\tag{$*$}\label{phieq0}
\end{equation}
Then
\[
\forall k \in \N\ \forall n \geq k  \left(\varphi_n \leq \frac{1}{1-\alpha} \varphi_k + \frac{1}{1-\alpha}\sum_{i = k}^{n-1} \delta_i - \sum_{i = k+1}^{n} s_i\right).
\]
\end{lemma}

\begin{remark}\label{GradLaraMarcaWrong}
Note that the statement of \cite[Lemma 2, (a)]{GradLaraMarca24} is slightly wrong: Under the assumptions of Lemma \ref{madness0}, the authors of \cite{GradLaraMarca24} claim that for all $k \geq 1$:
\begin{equation*}\label{wrongness}
\varphi_k + \sum_{i = 1}^{k}s_i \leq \varphi_0 + \frac{1}{1-\alpha}\sum_{i=0}^{k-1}\delta_i,
\end{equation*}
omitting the factor $\frac{1}{1-\alpha}$ in front of $\varphi_0$ as in the above Lemma \ref{madness0}. However, this factor can in general not be avoided. Indeed, consider e.g.\ $\varphi_k := 1$, $\alpha_k := 1/2$ for all $k\in\mathbb{N}$ and $s_k :=\delta_k := 1/3$ if $k = 0$ and $s_k :=\delta_k := 1/k^2$ otherwise. Then $\varphi_{k+1} - \varphi_k + s_k \leq \alpha_k (\varphi_k - \varphi_{k-1}) + \delta_k$ holds trivially for all $k \in \N$, but $\varphi_1 + s_1 \nleq \varphi_0 + \frac{1}{1-\alpha}\delta_0$. 
\end{remark}

Rearranging the inequality obtained in Lemma \ref{madness0}, we in particular get that $\sum_{i = 1}^{\infty}s_i<+\infty$, so that $\lim_{i\to\infty}s_i= 0$. The following lemma provides a quantitative variant of the corresponding (weaker) conclusion that $\liminf_{i\to\infty}s_i= 0$, which we require for the derivation of our rates later.

\begin{lemma}\label{madness}
Let $\{\varphi_k\}_k, \{s_k\}_k,\{\delta_k\}_k\subseteq [0,\infty)$, let $\varphi_{-1} \in [0,\infty)$ and let $\alpha \in [0,1)$ with $\{ \alpha_{k}\}_k \subseteq [0,\alpha]$. Assume that \eqref{phieq0} as in Lemma \ref{madness0} holds, and take $b \in [0,\infty)$ such that $b \geq \varphi_0$. Assume further that there exists $S \in (0,\infty)$ such that $\sum_{k = 0}^{\infty}\delta_k \leq S$. Then:
\[
\forall \varepsilon > 0\ \forall n\in \N\ \exists k \in \N\left(n \leq k \leq \left\lceil\frac{b + S}{(1-\alpha)\varepsilon} \right\rceil + n +1 \text{ and } s_{k+1} < \varepsilon\right).
\]
\end{lemma}

\begin{proof}
By Lemma \ref{madness0}, we get 
\begin{equation*}
\forall n \in\mathbb{N} \left(\sum_{i = 1}^{n}s_i \leq \frac{1}{1-\alpha}\varphi_0 + \frac{1}{1-\alpha}\sum_{i = 0}^{n-1}\delta_i\right)
\end{equation*}
so that taking $n \rightarrow \infty$ yields
\[
\sum_{i = 1}^{\infty}s_i \leq \frac{1}{1-\alpha}\varphi_0 + \frac{1}{1-\alpha}\sum_{i = 0}^{\infty}\delta_i\leq \frac{b+S}{1-\alpha}.
\]
Suppose now for a contradiction that the claim fails, i.e.\ there are $\varepsilon>0$ and $n\in\mathbb{N}$ such that $s_{k+1} \geq \varepsilon$ for all $k\in\mathbb{N}$ with $n \leq k \leq \left\lceil\frac{b + S}{(1-\alpha)\varepsilon}\right \rceil + n +1=:l$. Then however
\[
\sum_{i = 1}^{\infty}s_i \geq \sum_{i = n}^{l} s_{i+1} \geq \left(\left\lceil\frac{b + S}{(1-\alpha)\varepsilon }\right\rceil +1\right) \varepsilon > \frac{b+S}{1-\alpha},
\]
which is a contradiction to the above.
\end{proof}

Before we apply Lemma \ref{madness} to the sequences generated by \eqref{RIPPA}, we of course need to show that its conditions, in particular \eqref{phieq0} as in Lemma \ref{madness0}, will hold in our case. The next results are towards this.

The following proposition then extends \cite[Proposition 6]{GradLaraMarca24}. In difference to that result, we emphasize the slightly more general \ref{inertial_ineq_nonbasic} and \ref{inertial_ineq__long_nonbasic} below as we will later require varying $\lambda$'s akin to the approach taken in \cite{DespresPischke2026}, whereas only their instantiations with $\lambda = 1/2$ are considered explicitly in \cite[Proposition 6]{GradLaraMarca24}.

\begin{lemma}[{extending \cite[Proposition 6]{GradLaraMarca24}}]\label{GProposition6}
Let $f: K \times K \rightarrow \R$ be such that $(A2)$, $(A4)$ and $(A5)$ hold. Suppose that $x^{*} \in S(K,f)$. Then, for every $k \in \N$, at least one of the following is true:
\begin{align*}
d^2(x_{k+1},x^{*}) &\leq d^2(y_k,x^{*}) - \left(\frac{2 -\rho_k}{\rho_k}\right)d^2(x_{k+1},y_k) \tag*{$(+)_1$}\label{inertial_ineq_nonbasic}\\
&\hphantom{\leq\mbox{}}- \rho_k\left((1-\lambda)\gamma\beta_k - \lambda\right)d^2(z_k,x^{*})
\end{align*}
for all $\lambda \in [0,1]$ or
\begin{align*}
d^2(x_{k+1},x^{*}) &\leq d^2(y_k,x^{*}) - \left(\frac{2 -\rho_k-\frac{2 \beta_k\eta}{\lambda}}{\rho_k}\right)d^2(x_{k+1},y_k) \tag*{$(+)_2$}\label{inertial_ineq__long_nonbasic}\\
&\hphantom{\leq\mbox{}}-  \rho_k\left((1-\lambda)\gamma\beta_k - \lambda - \frac{2\eta\beta_k}{\lambda}\right) d^2(z_k, x^{*})
\end{align*}
for all $\lambda \in (0,1]$.
\end{lemma}

\begin{proof}
By definition, we have $x_{k+1} := \exp_{z_k}((1-\rho_k)\exp^{-1}_{z_k}y_k)$. Then it holds that
\[
d^2(x_{k+1},x^{*}) \leq  \rho_k d^2(z_k,x^*)+(1-\rho_k) d^2(y^k,x^*)-\rho_k(1-\rho_k)d^2(z_k,y_k).\tag*{$(\circ)_1$}\label{3a}
\]
Indeed, either $\rho_k\in [0,1]$, where then by Lemma \ref{CN} we get \ref{3a} immediately. Or $\rho_k\in [1,2]$, then $\rho_k=1+\rho_k'$ for some $\rho_k\in [0,1]$ so that $x_{k+1} = \exp_{z_k}-\rho_k'\exp^{-1}_{z_k}y_k$. Hence, in that case, Lemma \ref{revCN} yields
\begin{align*}
d^2(x_{k+1},x^{*}) &\leq (1+\rho_k') d^2(z_k,x^{*})-\rho_k' d^2(y_k,x^{*})+\rho_k'(1+\rho_k')d^2(z_k,y_k)\\
&= \rho_kd^2(z_k,x^{*})+(1-\rho_k) d^2(y_k,x^{*})-\rho_k(1-\rho_k)d^2(z_k,y_k),
\end{align*}
that is we also get \ref{3a} in this case. Rewriting \ref{3a}, we get 
\[
d^2(x_{k+1},x^{*}) \leq d^2(y^k,x^*) + \rho_k \left(d^2(z_k,x^*)- d^2(y^k,x^*)\right)-\rho_k(1-\rho_k)d^2(z_k,y_k).
\]
Further, by definition of the quasi-inner product, we get 
\[
2\langle\vv{y_kz_k},\vv{x^*y_k}\rangle+d^2(z_k,y_k)= d^2(z_k,x^*)-d^2(y_k,x^*)
\]
so that, combining both, we obtain
\[
d^2(x_{k+1},x^{*}) \leq d^2(y^k,x^*) +2\rho_k\langle\vv{y_kz_k},\vv{x^*y_k}\rangle+\rho^2_k d^2(z_k,y_k).\tag*{$(\circ)_2$}\label{3b}
\]
Lemma \ref{Lemma11P}, applied to $f(y_k,\cdot)$ with $y = x^{*}$, yields
\begin{equation}
f(y_k,z_k) - \mathrm{max}\{f(y_k,z_k),f(y_k,x^{*})\} \leq \frac{\lambda}{2}\left(\frac{\lambda}{\beta_k}-\gamma + \lambda\gamma\right)d^2(z_k, x^{*}) + \frac{\lambda}{\beta_k}\langle \vv{y_kz_k}, \vv{z_kx^*} \rangle\tag*{$(\circ)_3$}\label{3aa}
\end{equation}
for all $\lambda \in [0,1]$. We distinguish two cases.

\textbf{Case 1:} $f(y_k,z_k) \geq f(y_k,x^{*})$. From \ref{3aa}, we obtain
\begin{equation*}
\frac{\lambda}{\beta_k}\langle \vv{y_kz_k}, \vv{x^*z_k} \rangle\leq \frac{\lambda}{2}\left(\frac{\lambda}{\beta_k}-\gamma + \lambda\gamma \right)d^2(z_k,x^{*})
\end{equation*}
for any $\lambda\in [0,1]$. Then, proceed as in the next case from \ref{3aaa}, but where $\eta = 0$, to obtain \ref{inertial_ineq_nonbasic}. 

\textbf{Case 2:} $f(y_k,z_k) < f(y_k,x^{*})$. Using $(A5)$, we obtain from \ref{3aa} that 
\begin{align*}
0 &\leq \frac{\lambda}{2}\left(\frac{\lambda}{\beta_k}-\gamma + \lambda\gamma\right)d^2(z_k, x^{*}) + \frac{\lambda}{\beta_k}\langle \vv{y_kz_k}, \vv{z_kx^*} \rangle+f(y_k,x^{*})-f(y_k,z_k)\\
&\leq \frac{\lambda}{2}\left(\frac{\lambda}{\beta_k}-\gamma + \lambda\gamma\right)d^2(z_k, x^{*}) + \frac{\lambda}{\beta_k}\langle \vv{y_kz_k}, \vv{z_kx^*} \rangle+f(z_k,x^*)+\eta\left(d^2(z_k,y_k)+d^2(z_k,x^*)\right).
\end{align*}
Using the pseudomonotonicity of $f$ and since $x^*\in S(K,f)$, we get $f(z_k,x^*)\leq 0$ and so obtain
\[
\frac{\lambda}{\beta_k}\langle \vv{y_kz_k}, \vv{x^*z_k} \rangle \leq \frac{\lambda}{2}\left(\frac{\lambda}{\beta_k}-\gamma + \lambda\gamma+\frac{2\eta}{\lambda}\right)d^2(z_k, x^{*}) +\eta d^2(z_k,y_k)\tag*{$(\circ)_4$}\label{3aaa}
\]
for any $\lambda\in (0,1]$. But then, using $\langle \vv{y_kz_k}, \vv{x^*z_k} \rangle=\langle \vv{y_kz_k}, \vv{x^*y_k} \rangle+d^2(y_k,z_k)$, we obtain
\begin{equation*}
\langle \vv{y_kz_k}, \vv{x^*y_k} \rangle \leq \left(\frac{\eta\beta_k}{\lambda}-1\right) d^2(z_k,y_k) + \frac{1}{2}\left(\lambda-\gamma\beta_k + \lambda\gamma\beta_k + \frac{2\eta\beta_k}{\lambda}\right)d^2(z_k,x^{*}).
\end{equation*}
Using \ref{3b}, we obtain
\begin{align*}
d^2(x_{k+1},x^{*}) &\leq d^2(y^k,x^*) +\left(\frac{2\eta\beta_k\rho_k}{\lambda}-2\rho_k+\rho^2_k\right) d^2(z_k,y_k) \\
&\hphantom{\leq\mbox{}}-  \rho_k\left((1-\lambda)\gamma\beta_k - \lambda - \frac{2\eta\beta_k}{\lambda}\right)d^2(z_k,x^{*}).
\end{align*}
As $\gamma(t):= \exp_{z_k} (t\exp^{-1}_{z_k}y_k)$ is a geodesic line, we have
\[
d(x_{k+1},y_k)=d(\gamma(1-\rho_k),\gamma(1))=\rho_k d(z_k,y_k),
\]
which yields that $d(z_k,y_k)=\frac{1}{\rho_k}d(x_{k+1},y_k)$. Inserting this into the above, we get
\begin{align*}
d^2(x_{k+1},x^{*}) &\leq d^2(y_k,x^{*}) - \left(\frac{2 -\rho_k-\frac{2 \beta_k\eta}{\lambda}}{\rho_k}\right)d^2(x_{k+1},y_k)  \\ 
&\hphantom{\leq\mbox{}}-  \rho_k\left((1-\lambda)\gamma\beta_k - \lambda - \frac{2\eta\beta_k}{\lambda}\right) d^2(z_k ,x^{*})
\end{align*}
which is \ref{inertial_ineq__long_nonbasic}.
\end{proof}

From this, we almost immediately obtain the following lemma, which essentially is a nonlinear variant of \cite[Corollary 9]{GradLaraMarca24} with its case distinction resolved.

\begin{lemma}[{extending \cite[Corollary 9]{GradLaraMarca24}}]\label{inertail_easy_ineqs}
Let $f: K \times K \rightarrow \R$ be such that $(A2)$, $(A4)$ and $(A5)$ hold. Suppose that either among $(C1)^l$ or $(C1)^u$ hold. Suppose further that $x^{*} \in S(K,f)$. Then for every $k \in \N$:
\[
d^2(x_{k+1},x^{*}) \leq d^2(y_k,x^{*}) - \left(\frac{2 -\rho_k-4\eta\beta_k}{\rho_k}\right)d^2(x_{k+1},y_k).
\]
\end{lemma}
\begin{proof}
From Lemma \ref{GProposition6} we know that at least one of  \ref{inertial_ineq_nonbasic} or \ref{inertial_ineq__long_nonbasic} holds. Taking $\lambda=\frac{1}{2}$, the result then follows with the fact that, as $\beta_k\geq (1/(\gamma-8\eta)) \geq 0$ by either among $(C1)^l$ or $(C1)^u$,
\[
\frac{\rho_k(\gamma\beta_k-1)}{2} \geq \frac{\rho_k(\gamma\beta_k-1-8\eta\beta_k)}{2} = \frac{\rho_k(\beta_k(\gamma-8\eta)-1)}{2}  \geq 0
\]
and that further
\[
\frac{2-\rho_k}{\rho_k} \geq \frac{2 - \rho_k - 4\beta_k\eta}{\rho_k}
\]
for all $k \in \N$.
\end{proof}

This implies the following proposition, which we will use in our main theorem to argue that we are in the setting of Lemmas \ref{madness0} and \ref{madness}. 

\begin{lemma}[{extending \cite[Proposition 10]{GradLaraMarca24}}]\label{madness2}
Let $f: K \times K \rightarrow \R$ be such that $(A2)$, $(A4)$ and $(A5)$ hold. Suppose that either among $(C1)^l$ or $(C1)^u$ hold. Suppose further that $x^{*} \in S(K,f)$. Set $\varphi_k := d^2(x_k,x^{*})$ for every $k \in \N$. Then for every $k \in \N$:
\[
\varphi_{k+1} - \varphi_k + \left(\frac{2 -\rho_k-4\eta\beta_k}{\rho_k}\right)d^2(x_{k+1},y_k) \leq \alpha_k(\varphi_k - \varphi_{k-1}) +(\alpha_k^2 + \alpha_k)d^2(x_k, x_{k-1}).
\]
\end{lemma}
\begin{proof}
By definition, we have $y_k:=\exp_{x_k}(-\alpha_k\exp^{-1}_{x_k}x_{k-1})$. Using Lemma \ref{revCN}, we get 
\begin{align*}
d^2(y_k,x^*)&\leq d^2(x_k,x^*)+\alpha_k(d^2(x_k,x^*)-d^2(x_{k-1},x^*))+\alpha_k(1+\alpha_k)d^2(x_k,x_{k-1})\\
&=\varphi_k+\alpha_k(\varphi_k-\varphi_{k-1})+\alpha_k(1+\alpha_k)d^2(x_k,x_{k-1}).
\end{align*}
Combined with Lemma \ref{inertail_easy_ineqs}, this yields the result.
\end{proof}

As an application of Lemma \ref{madness0}, we now show that $\{x_k\}_k$ is quasi-Fej\'er monotone.

\begin{lemma}\label{almostfejer}
Let $f: K \times K \rightarrow \R$ be such that $(A2)$, $(A4)$ and $(A5)$. Suppose that $(C2)$ and either among $(C1)^l$ or $(C1)^u$ hold. Suppose further that $x^*\in S(K,f)$. Then:
\[
\forall k \in \N\ \forall n \geq k  \left(d^2(x_n,x^{*}) \leq \frac{1}{1-\alpha} d^2(x_k,x^{*}) + \frac{2}{1-\alpha}\sum_{i = k}^{n-1} \alpha_id^2(x_i, x_{i-1})\right).
\]
\end{lemma}
\begin{proof}
Define $\varphi_k := d^2(x_k,x^{*})$ and $\delta_k := (\alpha_k^2 + \alpha_k) d^2(x_k, x_{k-1})$ for all $k \in \N$. Observe that, using Lemma \ref{madness2}, condition \eqref{phieq0} in Lemma \ref{madness0} holds with 
\begin{equation*}
s_{k+1} := \left(\frac{2-4\eta\beta_k-\rho_k}{\rho_k} \right)d^2(x_{k+1},y_k)
\end{equation*}
as 
\begin{equation*}
\frac{2-4\eta\beta_k-\rho_k}{\rho_k} \overset{(C2)}{\geq} \frac{1-4\eta\beta_k-\rho}{1+\rho} \overset{\beta_k\leq e}{\geq} \frac{1-4\eta e-\rho}{1+\rho} \overset{(C2)}{\geq} 0.
\end{equation*}
Hence, using that $(\alpha_k^2 + \alpha_k)\leq 2\alpha_k$, Lemma \ref{madness0} yields the result.
\end{proof}

\subsection{Convergence under condition $(C1)^l$}\label{la}

We now establish the convergence of the method \eqref{RIPPA} under the condition $(C1)^l$, that is under the assumption that there exists a $\Theta>0$ such that $\frac{1}{\gamma-8\eta} \leq \beta_k \leq e-\Theta$ for all $k \in \N$, where $e\leq \upperbound$. The first result we need in that vein is the following lower bound on a key quantity featuring in later arguments.

\begin{lemma}\label{inetial_greater_half}
Suppose that $(C1)^l$ and $(C2)$ hold. Then, for any $k\in\mathbb{N}$:
\[
(2-4\eta\beta_k-\rho_k)\rho_k \geq \frac{16\eta^2\Theta}{(\gamma-8\eta)} > 0.
\]
\end{lemma}
\begin{proof}
Using $(C1)^l$ and $(C2)$, we have
\begin{align*}
(2-4\eta\beta_k-\rho_k)\rho_k &\geq (1-4\eta\beta_k-\rho )(1-\rho) \geq (4\eta e -4\eta\beta_k)4\eta e \\ 
&\geq (4\eta e -4\eta(e-\Theta))4\eta e = 16\eta^2 \Theta e.
\end{align*}
As we have $e \geq 1/(\gamma-8\eta)$, and as $\eta, \Theta>0$, this implies the result.
\end{proof}

Essential for our quantitative convergence theorem for \eqref{RIPPA} under $(C1)^l$ is an effective estimate for the asymptotic behavior of $d(z_k,y_k)$, which will here take the form of an explicit bound $\iota: (0, \infty)\times\mathbb{N} \rightarrow \N$ such that for any $\varepsilon>0$ and any $n\in\mathbb{N}$, there exists a $k\in\mathbb{N}$ with $n\leq k \leq \iota(\varepsilon,n)$ and $d(z_k,y_k) < \varepsilon$. The following result derives an explicit instantiation of such an $\iota$, by giving a quantitative analysis of \cite[Theorem 11, (a)]{GradLaraMarca24} in the context of $(C1)^l$.

\begin{lemma}\label{inertial_approx_fp}
Let $f: K \times K \rightarrow \R$ be such that $(A2)$, $(A4)$ and $(A5)$ hold. Suppose that $(C1)^l$ and $(C2)$ hold. Suppose further that $ x^{*} \in S(K,f)$. Let $S\in (0,\infty)$ be such that 
\[
\sum_{k = 0}^{\infty}\alpha_kd^2(x_k, x_{k-1}) \leq  S.
\]
Then
\[
\forall \varepsilon > 0\ \forall n \in \N\ \exists k \in \N \left(n \leq k \leq \left\lceil \frac{(b + 2S)(\gamma - 8\eta)}{16 \eta^2 \Theta (1-\alpha)\varepsilon^2}\right \rceil + n +1 \text{ and } d(z_k,y_k) < \varepsilon\right),
\]
where $b \geq d^2(x_0,x^{*})$.
\end{lemma}

\begin{proof}
Define $\varphi_k := d^2(x_k,x^{*})$, $\delta_k := (\alpha_k^2 + \alpha_k) d^2(x_k,x_{k-1})$ and \begin{equation*}
s_{k+1} := \left(\frac{2-4\eta\beta_k-\rho_k}{\rho_k} \right) d^2(x_{k+1},y_k).
\end{equation*}		
Note that it follows from $\alpha_k \in [0,1)$ that
\begin{equation*}
\sum_{k = 0}^{\infty} \delta_k = \sum_{k = 0}^{\infty}(\alpha_k^2 + \alpha_k)d^2(x_k,x_{k-1}) \leq 2\sum_{k = 0}^{\infty}\alpha_k d^2(x_k, x_{k-1}) \leq 2S.
\end{equation*}
Using Lemma \ref{madness2}, we get that Lemma \ref{madness} applies, so that for any $\varepsilon > 0$ and $n\in\mathbb{N}$, we get
\begin{equation*}
\exists k \in \N  \left(n \leq k \leq \left\lceil \frac{b + 2S}{ (1-\alpha)\varepsilon}\right \rceil + n+1  \text{ and } \left(\frac{2-4\eta\beta_k-\rho_k}{\rho_k}\right) d^2(x_{k+1},y_k) < \varepsilon\right).
\end{equation*}
As in Lemma \ref{GProposition6}, we have $\rho_kd(z_k,y_k)=d(x_{k+1},y_k)$, so that Lemma \ref{inetial_greater_half} implies that
\begin{equation*}
\left(\frac{16\eta^2\Theta}{(\gamma-8\eta)} \right) d^2(z_k,y_k) \leq \left(\frac{2-4\eta\beta_k-\rho_k}{\rho_k}\right) d^2(x_{k+1},y_k).
\end{equation*}
Combined with the above, this yields the result.
\end{proof}
		
\begin{remark}\label{rem:extraCond}
If in the above proposition, we impose the stronger property that 
\[
\sum_{k = 0}^{\infty}d^2(x_k, x_{k-1}) \leq  S'
\]
for some $S' \in (0,\infty)$, the resulting bound becomes simpler and can be derived in a more direct way (similar to the proof of \cite[Corollary 20]{GradLaraMarca24}). To outline the corresponding construction, assume that $(C2)$ and either among $(C1)^l$ or $(C1)^u$ hold. It is easy to see that for all $n,m \in \N$, there exists $k \in \N$ with $n \leq k \leq m+n$ and
\begin{equation*}
d(x_{k+1}, x_k)+ \alpha_k d(x_k, x_{k-1}) \leq \sqrt{\frac{8S'}{m}}.
\end{equation*}
Note that as $\gamma(t):= \exp_{x_k} (t\exp^{-1}_{x_k}x_{k-1})$ is a geodesic line, we have 
\[
d(y_k,x_k)=d(\gamma(-\alpha_k),\gamma(0))=\alpha_k d(x_k,x_{k-1}).
\]
Using that $\rho_kd(z_k,y_k)=d(x_{k+1},y_k)$ as in Lemma \ref{GProposition6}, we then have
\[
d(z_k, y_k) \leq \rho_k^{-1} d(x_{k+1}, y_k) \leq \rho_k^{-1}(d(x_{k+1},x_k)+\alpha_k d(x_k,x_{k-1})).
\]
Using the fact that by $(C2)$, and either among $(C1)^l$ or $(C1)^u$, we have that $\rho_k \geq \frac{4 \eta}{(\gamma-8\eta)}$ for all $k \in \N$, this implies that
\[
\forall\varepsilon>0\ \forall n \in \N\ \exists k \in \N \left(n \leq k \leq \left\lceil \frac{S'(\gamma-8\eta)^2}{2\eta^2 \varepsilon^2}\right\rceil+n+1\text{ and } d(z_k, y_k) < \varepsilon\right).
\]
A particular a priori assumption under which the above stronger property holds is condition $(C3)$ from \cite{GradLaraMarca24}, where it is assumed that there exists an $\overline{\alpha}=\xi/(2+\xi)$ such that 
\[
0\leq\alpha_k\leq\alpha_{k+1}\leq\alpha<\overline{\alpha}<1
\]
for all $k\in\mathbb{N}$, where
\[
\xi:=\frac{1-4\eta e-\rho}{1+\rho}>0.
\]
Concretely, following the same argument as given in the proof of \cite[Theorem 18]{GradLaraMarca24}, it can be shown that under condition $(C3)$ it holds that
\[
\sum_{k=0}^\infty d^2(x_{k+1},x_k)\leq\frac{1}{\xi-(2+\xi)\alpha}\cdot\frac{\mu_0}{1-\alpha}<+\infty.
\]
In particular, the above condition on the inertia parameters is a type of a priori condition which is sometimes easier to verify in practice than $\sum_{k=0}^\infty\alpha_k d^2(x_k,x_{k-1})<+\infty$, similar in spirit to the a priori conditions on the inertia parameters already pioneered by Alvarez and Attouch \cite{AlvarezAttouch01} (see also the seminal work of Bo\c{t}, Csetnek and Hendrich \cite{BotCsetnekHendrich2015} on splitting methods with inertia terms as well as \cite{SahuSharmaGautam2026} for similar discussions in the context of Hadamard manifolds).
\end{remark}
	
In our main convergence theorem, we will also utilize the following lemma which later allows us to transfer our effective convergence results for $\{x_k\}_k$ to the sequences $\{y_k\}_k$ and $\{z_k\}_k$.

\begin{lemma}\label{inertial_sum_good}
Let $f: K \times K \rightarrow \R$ be such that $(A2)$, $(A4)$ and $(A5)$ hold. Suppose that $(C1)^l$ and $(C2)$ hold. Suppose further that $x^{*} \in S(K,f)$. Let $M: (0,\infty) \rightarrow \N$ be such that
\[
\forall\varepsilon>0\left(\sum_{k = M(\varepsilon)}^{\infty}\alpha_kd^2(x_k,x_{k-1}) < \varepsilon\right).
\]
Then
\begin{equation*}
\forall\varepsilon>0\ \forall n \geq k \geq M\left(\frac{8 \eta^2 \Theta (1-\alpha) \varepsilon^2}{(\gamma-8\eta)}\right)  \left(d(z_{n} ,y_n)  < \sqrt{\frac{ (\gamma-8\eta)}{16 \eta^2 \Theta(1-\alpha) }}d(x_k,x^{*}) + \varepsilon\right).
\end{equation*}
\end{lemma}
\begin{proof}
We define $\varphi_k, \delta_k$ and $s_{k+1}$ as in Lemma \ref{almostfejer} and argue as therein to apply Lemma \ref{madness0}, by which we get
\begin{equation*}
\sum_{i = k+1}^{\infty}s_i \leq \frac{1}{1-\alpha}\varphi_k + \frac{1}{1-\alpha}\sum_{i = k}^{\infty}\delta_i
\end{equation*}
for any $k\in\mathbb{N}$. Hence, for $\varepsilon > 0$ and $k \geq  M((1-\alpha)\varepsilon^2/2)$, we obtain
\begin{equation*}
\sum_{i = k+1}^{\infty}\left(\frac{2-4\eta\beta_i-\rho_i}{\rho_i}\right) d^2(x_{i+1},y_i) < \frac{1}{1-\alpha}d^2(x_k, x^{*}) + \varepsilon.
\end{equation*} 
We can now argue similar as at the end of Lemma \ref{inertial_approx_fp} to derive
\begin{equation*}
\left(\frac{16\eta^2\Theta}{(\gamma-8\eta)} \right) d^2(z_i,y_i) \leq \left(\frac{2-4\eta\beta_i-\rho_i}{\rho_i}\right) d^2(x_{i+1},y_i).
\end{equation*}
This yields
\begin{equation*}
d^2(z_{n} ,y_n)  <\frac{ (\gamma-8\eta)}{16 \eta^2 \Theta(1-\alpha) }d^2(x_k, x^{*}) + \frac{ (\gamma-8\eta)}{16 \eta^2 \Theta }\varepsilon
\end{equation*}
for any $\varepsilon>0$ and any $n \geq k \geq  M((1-\alpha)\varepsilon/2)$. This yields the claim, using sub-additivity of the root.
\end{proof}

We now show our first main quantitative convergence result for the sequences $\{x_k\}_k$, $\{y_k\}_k$ and $\{z_k\}_k$ generated by \eqref{RIPPA}, set in the context of the assumption $(C1)^l$.

\begin{theorem}\label{THEOREM3}
Let $X$ be a Hadamard manifold and let $K\subseteq X$ be metrically closed such that the complete geodesic lines determined by any two points of $K$ lie in $K$. Let $f: K \times K \rightarrow \R$ be such that $(A2)$, $(A4)$ and $(A5)$ hold. Fix $\alpha, \rho \in [0,1)$ and $x_0, x_{-1} \in K$ as well as $\{ \beta_{k}\}_k \subseteq (0,\infty)$, $\{ \alpha_{k}\}_k \subseteq [0,\alpha]$ and $\{\rho_k\}_k \subseteq [1-\rho,1+\rho]$ such that $(C1)^l$ and $(C2)$ hold. Assume that \eqref{RIPPA} is well-defined and let $\{x_k\}_k, \{y_k\}_k$ and $\{z_k\}_k$ be its generated sequences. Suppose further that $S(K,f) \neq \emptyset$. If 
\[
\sum_{k=0}^\infty\alpha_k d^2(x_k,x_{k-1})<+\infty,
\]
then the sequences $\{x_k\}_k$, $\{y_k\}_k$ and $\{z_k\}_k$ converge to the (unique) equilibrium point $x^{*}$ of $f$. Further, we have the following rate of convergence: If $M: (0,\infty) \rightarrow \N $ and $S \in (0,\infty)$ are such that $M$ is decreasing and
\[
\sum_{k = 0}^{\infty}\alpha_kd^2(x_k, x_{k-1}) \leq  S\text{ as well as }\forall \varepsilon > 0\left(\sum_{k = M(\varepsilon)}^{\infty}\alpha_kd^2(x_k, x_{k-1}) < \varepsilon\right),
\]
then
\begin{equation*}
\forall \varepsilon > 0\ \forall k \geq \left(\left \lceil \frac{C}{\varepsilon^2} \right \rceil+ M(E \varepsilon^2) + 2\right) \left(d(x_k,x^{*}), d(y_k, x^{*}), d(z_k, x^{*}) < \varepsilon\right),
\end{equation*}
where 
\[
C := \mathrm{max}\left\{1, \frac{4 (3\gamma-28\eta)}{\eta^2 \Theta (1-\alpha)}\right\}\cdot\frac{256 (b+2S)(3\gamma-16\eta)^2}{\eta^2\Theta(1-\alpha)^2(3\gamma-28\eta)}
\]
and 
\[
E := \mathrm{min}\left\{1, \frac{\eta^2 \Theta (1-\alpha)}{4 (\gamma-8\eta)}\right\}\cdot\frac{(1-\alpha)}{32} 
\]
as well as $b \geq d^2(x_0, x^{*})$.
\end{theorem}
\begin{proof}
It is enough to show the quantitative result. Here, we first provide a rate for $\{x_k\}_k$. Given $\varepsilon >0$, using Lemma \ref{inertial_approx_fp}, take $k \in \N$ with
\[
M\left(\frac{(1-\alpha)\varepsilon^2}{32}\right)  \leq k \leq \left\lceil\frac{T}{\varepsilon^2} \right \rceil+ M\left(\frac{(1-\alpha)\varepsilon^2}{32}\right) +1
\]
such that
\[
d(z_k,y_k) < \frac{\sqrt{1-\alpha}}{64}\cdot\frac{(3\gamma - 28\eta)}{ (3\gamma-16\eta)}  \varepsilon < \frac{\sqrt{1-\alpha}}{32}\varepsilon,
\]
where
\[
T := \frac{ 64^2(b + 2S)(\gamma - 8\eta)(3 \gamma - 16 \eta)^2 }{16 \eta^2 \Theta (1-\alpha)^2(3\gamma-28 \eta)^2}.
\]
Note in particular that $k\leq \left \lceil C/\varepsilon^2\right \rceil + M(E \varepsilon^2)+1$. We now first establish the bound 
\begin{equation} 
d(z_k,x^{*}) < \frac{\sqrt{1-\alpha}}{8}\varepsilon.\tag{$-$}\label{eq3}
\end{equation}
For that we proceed by a case distinction similar to Lemma \ref{GProposition6}. First, using $z_k \in \mathrm{Prox}_{\beta_k f(y_k,\cdot)}( y_k )$ and applying Lemma \ref{Lemma11P} to $f(y_k,\cdot)$ with $y=x^*$ yields \ref{3aa} as in Lemma \ref{GProposition6}. Similar to Lemma \ref{GProposition6}, we consider two cases:

\textbf{Case 1:} $f(y_k,z_k) \geq f(y_k, x^*)$. From \ref{3aa}, we obtain
\[
0 \leq \frac{\lambda}{2}\left(\frac{\lambda}{\beta_k}-\gamma + \lambda \gamma\right)d^2(z_k,x^{*})^2 + \frac{\lambda}{\beta_k}\langle \vv{y_kz_k}, \vv{z_kx^*} \rangle
\]
for all $\lambda\in [0,1]$. Using \eqref{CS} and the fact that $\beta_k \geq \lowerbound$, the above implies
\[
\frac{1}{2}\left((1-\lambda)\frac{\gamma}{\gamma-8\eta}-\lambda\right)d^2(z_k,x^{*}) \leq d(z_k,y_k)d(z_k,x^*)
\]
for all $\lambda \in (0,1]$. Either $d(z_k,x^*)=0$, in which case \eqref{eq3} holds immediately, or $d(z_k,x^*)>0$, where we then get
\[
\frac{1}{2}\left((1-\lambda)\frac{\gamma}{\gamma-8\eta}-\lambda\right)d(z_k,x^{*}) \leq d(z_k,y_k)
\]
for all $\lambda \in (0,1]$. This in particular holds for $\lambda = \frac{1}{8}$ and we thus have
\[
\frac{3\gamma + 4 \eta}{8(\gamma - 8 \eta)}d(z_k,x^{*}) \leq d(z_k,y_k)< \frac{\sqrt{(1-\alpha)}}{32}\varepsilon < \frac{\sqrt{1-\alpha}}{8}\frac{3\gamma + 4 \eta}{8(\gamma - 8 \eta)} \varepsilon.
\]
so that we get \eqref{eq3} in that case.

\textbf{Case 2:}  $f(y_k,z_k) < f(y_k, x^*)$. As in Lemma \ref{GProposition6}, we obtain \ref{3aaa}, and so
\[
\frac{1}{2}\left((1-\lambda)\gamma\beta_k-\lambda- \frac{2\eta\beta_k}{\lambda}\right)d^2(z_k,x^{*}) \leq \frac{\eta\beta_k}{\lambda} d^2(z_k,y_k) +\langle \vv{y_kz_k}, \vv{z_kx^*} \rangle
\]
for all $\lambda\in (0,1]$. Using \eqref{CS}, we get
\[
\frac{1}{2}\left((1-\lambda)\gamma\beta_k-\lambda- \frac{2\eta\beta_k}{\lambda}\right)d^2(z_k,x^{*}) \leq \frac{\eta\beta_k}{\lambda} d^2(z_k,y_k) +d(z_k,y_k)d(z_k,x^*).
\]
We now make two further case distinctions.

\textbf{Case 2a:}  $d(z_k,x^{*})\leq d(z_k,y_k)$. As we in particular have $d(z_k,y_k) < \frac{\sqrt{1-\alpha}}{32}\varepsilon$, \eqref{eq3} follows immediately.

\textbf{Case 2b:} $d(z_k, x^{*})> d(z_k,y_k)$. This yields
\[
\frac{1}{2}\left((1-\lambda)\gamma\beta_k-\lambda- \frac{2\eta\beta_k}{\lambda}\right)d(z_k,x^{*}) \leq \left(1+\frac{\eta\beta_k}{\lambda} \right)d(z_k,y_k).
\]
for all $\lambda \in (0,1]$. We multiply both sides by $\frac{\lambda}{\lambda+\beta_k \eta}$ and rewrite the expression slightly to get
\begin{equation*}\label{sigh}
\frac{1}{2(\frac{\lambda}{\beta_k} +  \eta)}\left((\lambda(1-\lambda)\gamma-2\eta)-\frac{\lambda^2}{\beta_k}\right) d(z_k,x^{*})\leq d(z_k,y_k)
\end{equation*}
for all $\lambda \in (0,1]$. Setting $\lambda = \frac{3}{8}$ and using $\beta_k \geq 1/(\gamma-8\eta)$, we obtain
\begin{equation*}
(\lambda(1-\lambda)\gamma-2\eta)-\frac{\lambda^2}{\beta_k} \geq (\lambda (1-\lambda) \gamma - 2 \eta) - \lambda^2(\gamma-8\eta)  = \frac{6\gamma -56 \eta}{64} >0
\end{equation*}
since $12 \eta < \gamma$. Then further, again as $\beta_k \geq 1/(\gamma-8\eta)$, we get
\begin{equation*}
\frac{1}{2(\frac{\lambda}{\beta_k} +  \eta)}\Bigg((\lambda(1-\lambda)\gamma-2\eta)-\frac{\lambda^2}{\beta_k}\Bigg) \geq \frac{1}{2\bigg(\frac{3(\gamma-8\eta)}{8}+\eta\bigg)}\bigg(\frac{6\gamma-56 \eta}{64}\bigg)= \frac{(3\gamma - 28 \eta)}{8(3\gamma-16\eta)}.
\end{equation*}
Together, the above implies
\[
\frac{(3\gamma - 28 \eta)}{8(3\gamma-16\eta)}d(z_k, x^{*}) \leq d(z_k, y_k)< \frac{\sqrt{1-\alpha}}{8}\cdot\frac{(3\gamma - 28\eta)}{ 8(3\gamma-16\eta)} \varepsilon,
\]
so that we also get \eqref{eq3} in this case. As in Lemma \ref{GProposition6}, we have $\rho_kd(z_k,y_k)=d(x_{k+1},y_k)$. Since $\rho_k \leq 1+ \rho \leq 2$ by $(C2)$, using the triangle inequality we get
\begin{align*}
d(x_{k+1}, x^{*}) &\leq \rho_kd(z_k , y_k)+ d(y_k , z_k) + d(z_k, x^{*}) \\
&\leq 3d(z_k, y_k) + d(z_k,x^{*}) 
\end{align*}
With the bound on $d(z_k, y_k)$ and \eqref{eq3}, we thus obtain
\[
d(x_{k+1}, x^{*}) < \frac{3}{32} \sqrt{1-\alpha}\varepsilon + \frac{4}{32} \sqrt{1-\alpha}\varepsilon =  \frac{7}{32}\sqrt{1-\alpha} \varepsilon.
\]
But now, as in particular $k \geq M\left((1-\alpha)\varepsilon^2/32\right)$, Lemma \ref{almostfejer} yields
\begin{equation*}
d(x_n ,x^{*}) \leq \sqrt{\frac{1}{1-\alpha}}d(x_{k+1} , x^{*}) + \frac{\varepsilon}{4} < \frac{7}{32}\varepsilon + \frac{1}{4}\varepsilon = \frac{15}{32} \varepsilon<\varepsilon
\end{equation*}
for all $n \geq k+1$. This implies the rate for $\{x_k\}_k$ to $x^*$. We now turn to the rates for $\{z_k\}_k$ and $\{y_k\}_k$. Using triangle inequality, we get
\[
d(z_n,x^{*}) \leq d(z_n, y_n) + d(x_{n+1}, y_n)+ d(x_{n+1}, x^{*})
\]
for all $n\in\mathbb{N}$. Using $\rho_nd(z_n,y_n)=d(x_{n+1},y_n)$ and $\rho_n \leq 2$, we obtain
\[
d(z_n ,x^{*}) \leq 3 d(z_n, y_n) + d(x_{n+1}, x^{*}).
\]
Using Lemma \ref{inertial_sum_good}, we obtain
\begin{equation*}
\forall j \in \N\ \forall n \geq j \geq \left(M\left(\frac{ \eta^2 \Theta (1-\alpha)\varepsilon^2}{32 (\gamma-8\eta)}\right) +1\right) \left( d(z_n ,y_n)<\sqrt{\frac{ (\gamma-8\eta)}{16 \eta^2 \Theta(1-\alpha) }}d(x_j ,x^{*})+ \frac{\varepsilon}{16}\right).
\end{equation*}
As $M$ is decreasing, the above inequality holds for any $n\geq j$ such that
\[
j \geq j_0+1:=\left\lceil \frac{T}{\varepsilon^2} \frac{8^2(\gamma- 8 \eta)}{16 \eta^2 \Theta (1-\alpha)} \right \rceil + M\left(\frac{ \eta^2 \Theta (1-\alpha)^2\varepsilon^2}{2 \cdot 8^2 (\gamma-8\eta)}\right) +1.
\] 
For a $j\geq j_0+1$ as above, the above rate for $\{x_k\}_k$ to $x^*$ then in particular implies
\[
d(x_j, x^{*}) < \frac{\varepsilon}{16 \sqrt{\frac{\gamma-8\eta}{16 \eta^2 \Theta (1-\alpha)}}}
\] 
Combined, the above yields $d(z_n, y_n) <\varepsilon/8$ for all $n\geq j_0 +2$. Finally, this implies
\[
d(z_n, x^{*}) \leq 3 d(z_n, y_n) + d(x_{n+1}, x^{*}) < \frac{3}{8} \varepsilon + \frac{4}{8} \varepsilon = \frac{7}{8} \varepsilon
\]
for any $n\in\mathbb{N}$ such that
\[
n \geq \mathrm{max}\left\{\left\lceil \frac{T}{\varepsilon^2}\right\rceil+M\left(\frac{(1-\alpha)\varepsilon^2}{32}\right) , j_0 \right\} +2.
\]
Further, for an $n$ as above, a simple application of the triangle inequality yields
\[
d(y_n, x^{*}) \leq  d(z_n, y_n)+ d(z_n,x^{*}) < \frac{1}{8} \varepsilon + \frac{7}{8} \varepsilon = \varepsilon.
\]
This yields the rates for $\{z_k\}_k$ and $\{y_k\}_k$ after some slightly tedious but obvious simplifications.
\end{proof}

\begin{remark}\label{rem:meta}
Notice that the bound $S \in (0,\infty)$ on $\sum_{k = 0}^{\infty}\alpha_kd^2(x_k, x_{k-1})$ featuring in Theorem \ref{THEOREM3} above can in fact be computed from the rate $M: (0,\infty) \rightarrow \N$ assumed for the same series. Concretely, note that $\sum_{k = M(1)}^{\infty}\alpha_kd^2(x_k, x_{k-1}) < 1$ by the defining property of $M$, so that
\[
\sum_{k = 0}^{\infty}\alpha_kd^2(x_k, x_{k-1})\leq M(1)\cdot\max\{\alpha_kd^2(x_k, x_{k-1})\mid k<M(1)\}+1.
\]
Further, while a rate of convergence $M$ for the series $\sum_{k=0}^{\infty}\alpha_kd^2(x_k,x_{k-1}) < +\infty$ as in Theorem \ref{THEOREM3} above always exists, based on results from computability theory, it is generally not possible to obtain a computable such $M$ (see e.g.\ \cite{Kohlenbach2008}). However, it is often possible to obtain a computable so-called rate of metastability, that is a bound $M'(\varepsilon,g)$ such that
\[
\forall\varepsilon > 0\ \forall g: \N \rightarrow \N\ \exists N \leq M'(\varepsilon,g)\ \forall m,n \in [N; N+g(N)] \left(\sum_{k = m}^{n} \alpha_kd^2(x_k,x_{k-1}) < \varepsilon\right).
\]
Indeed, general logical metatheorems from the proof mining program \cite{Kohlenbach2008} (recall also the references in the introduction) guarantee that in an appropriate logical framework one can extract an effective $M'$ from a corresponding proof of $\sum_{k=0}^{\infty}\alpha_kd^2(x_k,x_{k-1}) < +\infty$. In fact, a construction thereof is already immediate from an upper bound $\sum_{k=0}^{\infty}\alpha_kd^2(x_k,x_{k-1}) \leq S$, where it follows from \cite[Proposition 2.27]{Kohlenbach2008} (see also \cite[Remark 2.29]{Kohlenbach2008}) that $M'(\varepsilon,g):=\tilde{g}^{(\lceil S\cdot \varepsilon^{-1}\rceil)}(0)$ with $\tilde{g}(n):=n+g(n)$ is a rate of metastability. In particular, such a rate can hence immediately be given in the context of the parameter conditions from Remark \ref{rem:extraCond}. 

While the previous arguments can be adapted to this computationally weaker assumption of a rate of metastability for the series, it would then naturally restrict us to only be able to obtain rates of metastability instead of full rates of convergence in Theorem \ref{THEOREM3} above (as well as Theorem \ref{THEOREM3a} below), which is not the focus of this paper. While computable rates of convergence for $M$ are generally difficult to obtain, there are common online choices for $\alpha_k$ (already pioneered by Alvarez and Attouch \cite{AlvarezAttouch01}) where an explicit and effective rate $M$ for $\sum_{k=0}^{\infty}\alpha_kd^2(x_k,x_{k-1}) < +\infty$ can be directly given. Concretely, consider the online choice discussed in \cite[Remark 12]{GradLaraMarca24}:
\[
\alpha_k:=\min\left\{\alpha,\frac{\theta^k}{d^2(x_k,x_{k-1})}\right\}
\]
for some $\theta\in (0,1)$, supposing (w.l.o.g.) that $x_k\neq x_{k-1}$ for all $k\in\mathbb{N}$. Then one has 
\[
\sum_{k=n}^{\infty}\alpha_kd^2(x_k,x_{k-1}) \leq \sum_{k=n}^{\infty}\theta^k=\frac{1}{1-\theta}-\sum_{k=0}^{n-1}\theta^k= \frac{\theta^{n}}{1-\theta}
\]
so that $M(\varepsilon):=\frac{\log(\varepsilon(1-\theta))}{\log(\theta)}+1$ is a corresponding rate of convergence.
\end{remark}

\begin{remark}\label{rem:generalization}
In the context of assumption $(C1)^l$, the quantitative result contained in Theorem \ref{THEOREM3} is, to our knowledge, already novel in the Euclidean case. In terms of assumptions on the bifunction, we find that compared to \cite{GradLaraMarca24}, $(A3)$ can be dropped completely. Further, if $(A1)$ is added as an assumption (which, compared to \cite{GradLaraMarca24}, here crucially contains the lower semicontinuity of the bifunction in its right argument in the absence of $(A3)$), then the assumptions that $S(K,f) \neq \emptyset$ and that \eqref{RIPPA} is well-defined can both be dropped from Theorem \ref{THEOREM3}, using \cite[Theorem 3.9]{DespresPischke2026} and \cite[Proposition 4.4]{DespresPischke2026}, respectively.
\end{remark}

\begin{remark}\label{rem:Hilbert}
As highlighted in the introduction, since we do not rely on the finite dimensionality of the manifold in any way, the result extends mutatis mutandis to suitable Hadamard spaces which satisfy all the required metric properties outlined in Section \ref{sec:HadamardMani}. Concretely, the results given here already hold over Hadamard spaces $(X,d)$ which satisfy the following two properties: 
\begin{enumerate}
\item For each pair of distinct points $x,y\in X$, there is a unique metric line $r:\mathbb{R}\to X$ which passes through $x$ and $y$.
\item For any $x,y,w\in X$ and $t\in [0,1]$, for the unique point $z$ on the metric line passing through $x$ and $y$ such that $d(z,x)=td(x,y)$ and $d(z,y)=(1+t)d(x,y)$, it holds that
\[
d^2(z,w) \leq  (1-t)d^2(x,w)+t d^2(y,w)-t(1-t)d^2(x,y).
\]
\end{enumerate}
Spaces closed under metric lines as in item (1) above are also key for the notion of hyperbolic space studied by Reich and Shafrir \cite{ReichShafrir1990}. Writing $(1+t)x\ominus ty$ for the unique point $z$ on the metric line passing through $x$ and $y$ such that $d(z,x)=td(x,y)$ and $d(z,y)=(1+t)d(x,y)$ and $(1-t)x\oplus ty$ for the unique point $z$ on the metric line passing through $x$ and $y$ such that $d(z,x)=td(x,y)$ and $d(z,y)=(1-t)d(x,y)$, we can see that
\[
\exp_{x} (t\exp^{-1}_{x}y)=(1-t)x\oplus ty\text{ and }(1+t)x\ominus ty=\exp_{x} (-t\exp^{-1}_{x}y)
\]
for all $x,y\in X$ and $t\in [0,1]$. Lemma \ref{revCN} thereby shows that Hadamard manifolds have the property from item (2) above. Lemma \ref{CN} holds just as well for general Hadamard spaces if we take $z=(1-t)x\oplus ty$. Next to Hadamard manifolds, another class of spaces that satisfies the above properties are Hilbert spaces. Now, with that notation the method \eqref{RIPPA} can be defined over such Hadamard spaces via
\[
\begin{cases}
y_k := (1+\alpha_k)x_k\ominus \alpha_k x_{k-1},\\
z_{k} \in \argmin_{x\in K}\left\{f(y_k,x)+ \frac{1}{2\beta_k}d^2(y_k,x)\right\},\\
x_{k+1} := \begin{cases} \rho_k z_k\oplus (1-\rho_k)y_k, &\text{if }\rho_k\in [1-\rho,1],\\
\rho_k z_k\ominus (\rho_k-1)y_k, &\text{if }\rho_k\in [1,1+\rho].
\end{cases}
\end{cases}
\]
for similar parameter restrictions as before, and the above results extend suitably to this purely metric variant, and hence in particular to Hilbert spaces (a detailed derivation tailored to that context can be found in the master thesis of the first author \cite{Despres2026}). In that case, the rates remain exactly the same. Further, the convergence is then still strong, that is w.r.t.\ the norm and not relative to any notion of weak convergence. Akin to \cite{DespresPischke2026,Pischke2025}, this is also here essentially due to the fact that equilibrium problems for strongly quasiconvex and pseudomonotone bifunctions have a unique solution.
\end{remark}

\begin{remark}\label{rem:quasiFejer}
The quantitative result contained in Theorem \ref{THEOREM3} above could also have been obtained by an application of the general results for quasi-Fej\'er monotone sequences under general metric regularity assumptions as developed by the second author in \cite{Pischke2023} (extending the work of Kohlenbach, L\'opez-Acedo and Nicolae \cite{KohlenbachLN2019}). We refer to the master thesis of the first author \cite{Despres2026} for further discussions in that vein.
\end{remark}

\subsection{Convergence under condition $(C1)^u$}

We now establish the convergence of the method \eqref{RIPPA} under the condition $(C1)^u$, that is under the assumption that there exists a $\theta>0$ such that $\lowerbound + \theta \leq \beta_k \leq e$ for all $k \in \N$, where $e \leq \upperbound$.

Similar to before, an essential ingredient for our quantitative convergence theorem for \eqref{RIPPA} under this assumption is an effective estimate for the asymptotic behavior of $d(z_k,y_k)$, again taking the form of an explicit bound $\iota: (0, \infty)\times\mathbb{N} \rightarrow \N$ such that for any $\varepsilon>0$ and any $n\in\mathbb{N}$, there exists a $k\in\mathbb{N}$ with $n\leq k \leq \iota(\varepsilon,n)$ and $d(z_k,y_k) < \varepsilon$. 

To motivate how this is achieved under the assumption $(C1)^u$, first recall how we have proceeded in the analogous result for $(C1)^l$, that is Lemma \ref{inertial_approx_fp}: Using \ref{inertial_ineq__long_nonbasic} from Lemma \ref{GProposition6} with $\lambda=\frac{1}{2}$, we suitably instantiated \eqref{phieq0} from Lemma \ref{madness0}, after which we were able to infer the existence of a point $k \geq n$ (together with a respective bound) where 
\[
\left(\frac{2-4\eta\beta_k-\rho_k}{\rho_k}\right) d^2(x_{k+1},y_k)
\]
is suitably small. While $(C1)$ and $(C2)$ as in \cite{GradLaraMarca24} allow for $(2-4\eta\beta_k-\rho_k)/\rho_k \rightarrow 0$ (setting $\rho_k = 2-4\eta e$ for all $k \in \N$ and considering $\beta_k \rightarrow e$), the upper bound $\Theta >0$ such that $\beta_k \leq e-\Theta $ for all $k \in \N$ contained in $(C1)^l$ prohibits this behavior, so that in that context $(2-4\eta\beta_k-\rho_k)\rho_k$ can be bounded uniformly away from zero so that it can be ensured that also $\rho_k^{-2}d^2(x_{k+1},y_k) =d^2(z_k, y_k)$ is suitably small at this point. In that context, one can then even generalize the lower bound on $\{\beta_k\}_k$ to $\beta_k \geq \lowerbound$ for all $k \in \N$.

Now, we can actually treat the non-strict upper bound $\beta_k \leq \upperbound$ by using \ref{inertial_ineq__long_nonbasic} from Lemma \ref{GProposition6} with $\lambda$ chosen adaptively as $\lambda_k$ with $k$ (or rather, the values of $\beta_k$) instead of setting $\lambda=\frac{1}{2}$. Using an appropriate $\lambda_k$, we will be able to infer the existence of a point $k\geq n$ (together with a respective bound) where 
\[
\left(\frac{2 -\rho_k-\frac{2 \beta_k\eta}{\lambda_k}}{\rho_k}\right) d^2(x_{k+1},y_k)
\]
is suitably small. But to now simultaneously ensure that we are able to uniformly bound $\left(2 -\rho_k-2 \beta_k\eta/\lambda_k\right)\rho_k$ away from zero for our adaptive $\lambda_k$, $\beta_k \rightarrow \lowerbound$ must be excluded, leading to the assumption $(C1)^u$ which introduces a strict lower bound $\theta>0$ such that $\lowerbound + \theta \leq \beta_k$ for all $k \in \N$. Once this is guaranteed, we are able to infer that also $\rho_k^{-2}d^2(x_{k+1},y_k) =d^2(z_k, y_k)$ is suitably small, similar to before.

To now choose these $\lambda_k$ adaptively with $\beta_k$, we will rely on some technical machinery introduced for a similar purpose in \cite{DespresPischke2026}.\footnote{The strategy of choosing $\lambda_k$ adaptively with $\beta_k$ used here was first used in \cite{DespresPischke2026} to extend the parameter restrictions of the proximal point type method from \cite{IusemLara2021} (and its nonlinear version studied in \cite{DespresPischke2026}) from the open interval $(\lowerbound,\upperbound)$ to the closed interval $[\lowerbound,\upperbound]$.} For $x \geq \lowerbound$, set
\begin{equation}
N(x)  := \frac{\gamma x}{2(1+\gamma x)} + \frac{\sqrt{\gamma^2 x^2 - 8 x\eta(1+\gamma x)}}{2(1+\gamma x)},\tag{$\dagger$}\label{N}
\end{equation}
where $\gamma$ and $\eta$ are as in $(A4)$ and $(A5)$. Notice that $N$ is well-defined on $[\lowerbound,\infty)$, as for $x \geq \lowerbound$ the value $\gamma^2x^2 - 8x\eta(1+\gamma x)$ is positive. The key properties of $N$ which we require are collected in the following lemma. All items apart from item (v) already appear with corresponding proofs in \cite{DespresPischke2026}, so that we only provide details for the last item here.

\begin{lemma}[{extending \cite[Lemma 4.11]{DespresPischke2026}}]\label{Nporps}
Consider $N$ as defined in \eqref{N}. Then, we have the following:
\begin{enumerate}
\item[(i)] $N(\lowerbound) = \frac{1}{2},$
\item[(ii)] $N(\upperbound) > \frac{4\gamma}{3(8\eta+2\gamma)} > \frac{1}{2}$,
\item[(iii)] $N: [\lowerbound, \infty) \rightarrow [0,1]$ is increasing,
\item[(iv)] $ x \mapsto \frac{N(x)}{2x}$ is decreasing for $x \in [\lowerbound, \upperbound]$,
\item[(v)] For any $\theta > 0$:
\[
N\left(\lowerbound + \theta\right) \geq \frac{1}{2-\frac{2\theta(\gamma-8\eta)^2}{(1+\theta(\gamma-8\eta))(2\gamma-8\eta)}}.
\]
\end{enumerate}
\end{lemma}
\begin{proof}
We only prove (v), the proofs for the remaining items can be found in the proof of \cite[Lemma 4.11]{DespresPischke2026}. Note first that a simple calculation yields
\[
N(x) = \frac{\gamma}{\frac{2}{x}+2\gamma}+ \frac{\sqrt{\gamma^2-\frac{8\eta}{x}-8\gamma\eta}}{\frac{2}{x}+2\gamma}.
\]
Further, we have $\sqrt{\gamma(\gamma-8\eta)-\frac{8\eta}{x}} \geq (\gamma-8\eta)$ for all $x \geq \lowerbound$. Together, this yields
\begin{align*}
N\left(\lowerbound+\theta \right) &\geq \frac{\gamma+(\gamma-8\eta)}{\frac{2}{\frac{1}{\gamma-8\eta}+\theta}+2\gamma}= \frac{2\gamma - 8 \eta }{\frac{2(\gamma-8\eta)}{1+\theta (\gamma - 8 \eta)}+ 2 \gamma} \\
&= \frac{(2\gamma - 8\eta)(1+\theta(\gamma - 8 \eta))}{2(2\gamma - 8\eta)+2\gamma \theta (\gamma - 8\eta)} = \frac{1 + \theta(\gamma - 8 \eta)}{2 + \frac{2\gamma \theta(\gamma-8\eta)}{2\gamma - 8 \eta}}.
\end{align*}
For shorter notation, set $t := \theta(\gamma - 8 \eta)$. We continue to observe that  
		\begin{align*}
			\frac{2}{1+ t} + \frac{2 \gamma t}{(1+t)(2\gamma-8\eta)} &= 2 - \left(2-\frac{2}{1+t} - \frac{2\gamma t}{(1+t)(2\gamma - 8\eta)}\right)\\[.5em]
			& = 2- \left(\frac{2(1+t)(2\gamma - 8 \eta)-2 (2\gamma - 8 \eta)-2\gamma t}{(1+t)(2\gamma - 8\eta)}\right)\\[.5em]
			& = 2- \left(\frac{2t(2\gamma - 8 \eta)-2\gamma t}{(1+t)(2\gamma - 8\eta)}\right)\\[.5em]
			&= 2- \left(\frac{2 t (\gamma - 8 \eta)}{(1+t)(2\gamma-8\eta)}\right).
		\end{align*}
		Finally, combining the above, we obtain the result:
		\begin{align*}
			N\left(\lowerbound + \theta \right) \geq \frac{1 + \theta(\gamma - 8 \eta)}{2 + \frac{2\gamma \theta(\gamma-8\eta)}{2\gamma - 8 \eta}} = \frac{1}{\frac{2}{(1+\theta(\gamma - 8\eta))}+ \frac{2\gamma\theta(\gamma-8\eta)}{(1+\theta(\gamma-8\eta))(2\gamma - 8 \eta)}} = \frac{1}{2 - \frac{2\theta(\gamma - 8 \eta)^2}{(1 + \theta(\gamma - 8\eta))(2\gamma - 8 \eta)}}.
		\end{align*}
\end{proof}

The function $N$ can now be used to introduce our adaptive parameter choices $\lambda_k$ and the following lemma provides both the definition and essential property related to this. A variant of this result, featuring a slightly refined $\lambda_k$, already appears in \cite[Lemma 4.12]{DespresPischke2026}, and while the proof is essentially the same, we provide the details here for self-containedness.

\begin{lemma}\label{collection}
Suppose that $12 \eta < \gamma$ holds. For any $k \in \N$, let $\beta_k \in [\lowerbound, \upperbound]$ and consider
\[
c_k := \frac{\lambda_k}{2\beta_k}, \qquad
	d_k := \frac{\lambda_k(1-\lambda_k)(1+\gamma\beta_k)}{2\beta_k}-\eta
\]
for $\lambda_k:=N(\beta_k)$, where $N$ is defined as in \eqref{N}. Then for any $k \in \N$: $c_k - d_k = 0$.
\end{lemma}
\begin{proof}
We have 
\[
c_k - d_k = \frac{\lambda_k - \lambda_k(1-\lambda_k)(1+\gamma\beta_k)}{2\beta_k}+\eta.
\]
To see that $c_k - d_k = 0$, notice that
\[
c_k - d_k = 0 \Leftrightarrow \lambda_k^2 - \lambda_k\frac{\gamma\beta_k}{(1+\gamma\beta_k)}+\frac{2\beta_k\eta}{(1+\gamma\beta_k)} = 0.
\]
By definition of $N$, the right hand side now holds by the quadratic formula.
\end{proof}

Having introduced our adaptive convex combination parameters $\lambda_k$, we also here rely on a specific lower bound on a key quantity featuring in later arguments, similar to Lemma \ref{inetial_greater_half}.

\begin{lemma}\label{inertial_greater_lambda_k}
Suppose that $(C1)^u$ and $(C2)$ hold. Then, for all $k \in \N$:
\[
\left(2 -\rho_k-\frac{2 \beta_k}{N(\beta_k)}\eta\right)\rho_k\geq \frac{16\eta^2\theta}{(1+\theta(\gamma - 8 \eta))(2\gamma - 8 \eta)} > 0.
\]
\end{lemma}
\begin{proof}
Using the fact that $N$ is increasing by Lemma \ref{Nporps}, (iii) yields
\[
\left(2 -\rho_k-\frac{2 \beta_k}{N(\beta_k)}\eta\right) \rho_k \overset{(C2)}{\geq} \left(4 \eta e - \frac{2\beta_k}{N(\beta_k)}\eta \right) 4 \eta e\overset{(C1)^u}{\geq}\left(4 \eta e - \frac{2e\eta}{N(1/(\gamma-8\eta) + \theta)}\right)4 \eta e.
\]
Now using Lemma \ref{Nporps}, (v), and since $e \geq 1/(\gamma-8\eta)$ by $(C1)^u$, we obtain
\begin{align*}
\left(2 -\rho_k-\frac{2 \beta_k}{N(\beta_k)}\eta\right)\rho_k &\geq
\left(4 \eta e - 2 \eta e \left( 2- \frac{2\theta (\gamma - 8 \eta)^2}{(1+ \theta(\gamma-8\eta))(2\gamma-8\eta)}\right)\right) 4 \eta e\\
&= 16\eta^2 e^2 \frac{\theta (\gamma-8\eta)^2}{(1+\theta(\gamma-8\eta))(2\gamma - 8\eta)}\\
&\geq \frac{16 \eta ^2 \theta }{(1+\theta(\gamma-8\eta))(2\gamma - 8 \eta)}.
\end{align*}  
Notice that the last term is greater than $0$ by $(C1)^u$.
\end{proof}

We now provide the following counterpart to Lemma \ref{inertail_easy_ineqs}.

\begin{lemma}\label{inertail_easy_ineqs_lamda_k}
Let $f: K \times K \rightarrow \R$ be such that $(A2)$, $(A4)$ and $(A5)$ hold. Suppose that either among $(C1)^l$ or $(C1)^u$ hold. Suppose further that $x^{*} \in S(K,f)$. Then, for every $k \in \N$:
\[
d^2(x_{k+1},x^{*}) \leq d^2(y_k,x^{*}) - \left(\frac{2 -\rho_k-\frac{2 \beta_k\eta}{N(\beta_k)}}{\rho_k}\right)d^2(x_{k+1},y_k).
\]
where $N$ is defined as in (\ref{N}).
\end{lemma}
\begin{proof}
From Lemma \ref{GProposition6} we know that for all $k\in\mathbb{N}$ that at least one of \ref{inertial_ineq_nonbasic} or \ref{inertial_ineq__long_nonbasic} holds, that is (using Lemma \ref{Nporps}, (iii)), that either
\begin{equation*}
d^2(x_{k+1},x^*) \leq d^2(y_k,x^*) - \left(\frac{2 -\rho_k}{\rho_k}\right)d^2(x_{k+1},y_k) - \rho_k\left((1-N(\beta_k))\gamma\beta_k - N(\beta_k)\right)d^2(z_k,x^*)
\end{equation*}
or
\begin{align*}
d^2(x_{k+1},x^*) &\leq d^2(y_k,x^*) - \left(\frac{2 -\rho_k-\frac{2 \beta_k\eta}{N(\beta_k)}}{\rho_k}\right)d^2(x_{k+1},y_k)  \\
&\hphantom{\leq\mbox{}}- \rho_k\left((1-N(\beta_k))\gamma\beta_k - N(\beta_k) - \frac{2\eta\beta_k}{N(\beta_k)}\right) d^2(z_k, x^*).
\end{align*}
We begin by arguing that, in both of the above inequalities, we can discard the last term. As $\rho_k \geq 0$, we do this by showing that
\[
\left((1-N(\beta_k))\gamma\beta_k - N(\beta_k)\right) \geq \left((1-N(\beta_k))\gamma\beta_k - N(\beta_k) - \frac{2\eta\beta_k}{N(\beta_k)}\right) \geq 0.
\]
The first inequality is obvious. The second is (by multiplying with $\frac{N(\beta_k)}{2\beta_k}$) equivalent to showing 
\begin{align*}
\left(\frac{N(\beta_k)(1-N(\beta_k))(1+\gamma\beta_k) }{2\beta_k}- \frac{N(\beta_k)}{2 \beta_k} - \eta\right) \geq 0.
\end{align*}
But now 
\begin{align*}
\left(\frac{N(\beta_k)(1-N(\beta_k))(1+\gamma\beta_k) }{2\beta_k}- \frac{N(\beta_k)}{2 \beta_k} - \eta\right)  = d_k - c_k,
\end{align*}
where $c_k $ and $d_k$ are as in Lemma \ref{collection}, which implies that $d_k-c_k= 0$ holds. The inequality now follows from the fact that 
\begin{equation*}
\left(\frac{2 -\rho_k-\frac{2 \beta_k\eta}{N(\beta_k)}}{\rho_k}\right) \leq \frac{2 -\rho_k}{\rho_k}.\qedhere
\end{equation*}
\end{proof}

In the same way as Lemma \ref{madness2}, we also obtain the following result:

\begin{lemma}[{extending \cite[Proposition 10]{GradLaraMarca24}}]\label{inertial_fejer_lamda_k}
Let $f: K \times K \rightarrow \R$ be such that $(A2)$, $(A4)$ and $(A5)$ hold. Suppose that either among $(C1)^l$ or $(C1)^u$ hold. Suppose further that $x^{*} \in S(K,f)$. Set $\varphi_k := d^2(x_k,x^{*})$ for every $k \in \N$. Then, for every $k \in \N$:
\[
\varphi_{k+1} - \varphi_k -\alpha_k(\varphi_k - \varphi_{k-1}) + \left(\frac{2 -\rho_k-\frac{2 \beta_k\eta}{N(\beta_k)}}{\rho_k}\right)d^2(x_{k+1},y_k) \leq (\alpha_k^2 + \alpha_k)d^2(x_k, x_{k-1})
\]
where $N$ is defined as in (\ref{N}).
\end{lemma}

We continue with the analogue of Lemma \ref{inertial_approx_fp}.

\begin{lemma}\label{inertial_approx_fp_N}
Let $f: K \times K \rightarrow \R$ be such that $(A2)$, $(A4)$ and $(A5)$ hold. Suppose that $(C1)^u$ and $(C2)$ hold. Suppose further that $x^{*} \in S(K,f)$. Let $S\in (0,\infty)$ be such that 
\[
\sum_{k = 0}^{\infty}\alpha_kd^2(x_k, x_{k-1}) \leq  S.
\]
Then
\begin{equation*}
\forall \varepsilon>0\ \forall n \in \N\ \exists k \in \N \left( n \leq k \leq \left\lceil \frac{(b + 2S) ((1+\theta (\gamma-8\eta))(2\gamma - 8\eta)}{16 \eta^2 (1-\alpha) \theta \varepsilon^2}\right \rceil + n +1 \text{ and } d(z_k,y_k)< \varepsilon\right),
\end{equation*}
where $b \geq d^2(x_0,x^{*})$.
\end{lemma}
\begin{proof}
Proceeding as in the proof of Lemma \ref{inertial_approx_fp}, define $\varphi_k := d^2(x_k,x^{*})$ and $\delta_k := (\alpha_k^2 + \alpha_k) d^2(x_k, x_{k-1})$ for all $k\geq 0$. We have
\begin{equation*}
s_{k+1} := \left(\frac{2 -\rho_k-\frac{2 \beta_k\eta}{N(\beta_k)}}{\rho_k}\right) d^2(x_{k+1},y_k)\geq 0
\end{equation*}
by Lemma \ref{inertial_greater_lambda_k} and $\sum_{k = 0}^{\infty} \delta_k \leq 2S$. Using Lemma \ref{inertial_fejer_lamda_k}, we get that Lemma \ref{madness} applies so that for $\varepsilon>0$ and $n\in\mathbb{N}$ as before, we have 
\[
\exists k \in \N \left( n \leq k \leq \left\lceil \frac{b + 2S}{(1-\alpha)\varepsilon}\right \rceil + n +1\text{ and } \left(\frac{2 -\rho_k-\frac{2 \beta_k\eta}{N(\beta_k)}}{\rho_k}\right) d^2(x_{k+1},y_k) < \varepsilon\right).
\]
Using that $\rho_kd(z_k,y_k)=d(x_{k+1},y_k)$, this implies that 
\begin{equation*}
\left(2 -\rho_k-\frac{2 \beta_k\eta}{N(\beta_k)}\right) \rho_k  d^2(z_k,y_k) < \varepsilon.
\end{equation*}
With Lemma  \ref{inertial_greater_lambda_k}, we obtain that 
\begin{equation*}
\frac{16\eta^2 \theta}{(1+\theta(\gamma - 8 \eta))(2\gamma - 8 \eta)} d^2(z_k,y_k) < \varepsilon
\end{equation*}
for such a $k$. Put together, this yields the claim.
\end{proof}

Towards showing a quantitative convergence result not only for $\{x_k\}_k$, but also for $\{y_k\}_k$ and $\{z_k\}_k$, we establish an analogous result to Lemma \ref{inertial_sum_good}.

\begin{lemma}\label{inertial_sum_good_N}
Let $f: K \times K \rightarrow \R$ be such that $(A2)$, $(A4)$ and $(A5)$ hold. Suppose that $(C1)^u$ and $(C2)$ hold. Suppose further that $x^{*} \in S(K,f)$. Let $M: (0,\infty) \rightarrow \N$ be such that
\[
\forall\varepsilon>0\left( \sum_{k = M(\varepsilon)}^{\infty}\alpha_kd^2(x_k,x_{k-1}) < \varepsilon\right).
\]
Then:
\begin{gather*}
\forall\varepsilon>0\ \forall k \in \N\forall n > k \geq  M\left(\frac{8 \eta^2 \theta (1-\alpha) \varepsilon^2}{(1+\theta(\gamma-8\eta))(2\gamma-8\eta)}\right)  \\ \left(d(z_{n}, y_n) < \sqrt{\frac{ (1+\theta(\gamma-8\eta))(2\gamma-8\eta)}{16 \eta^2 \theta(1-\alpha) }}d(x_k, x^{*}) + \varepsilon\right).
\end{gather*}
\end{lemma}
\begin{proof}
We argue as in Lemma \ref{inertial_sum_good} to obtain that
\begin{equation*}
\sum_{i = k+1}^{\infty}\left(\frac{2 -\rho_i-\frac{2 \beta_i\eta}{N(\beta_i)}}{\rho_i}\right)d^2(x_{i+1},y_i) < \frac{1}{1-\alpha}d^2(x_k, x^{*}) + \varepsilon.
\end{equation*} 
for any $\varepsilon > 0$ and any $k \geq  M((1-\alpha)\varepsilon/2)$. Proceeding as in the proof of Lemma \ref{inertial_sum_good}, using $\rho_kd(z_k,y_k)=d(x_{k+1},y_k)$ and Lemma \ref{inertial_greater_lambda_k}, yields the result.
\end{proof}

The following is now our second main quantitative convergence result for the sequences $\{x_k\}_k$, $\{y_k\}_k$ and $\{z_k\}_k$ generated by \eqref{RIPPA}, set in the context of the assumption $(C1)^u$.

\begin{theorem}\label{THEOREM3a}
Let $X$ be a Hadamard manifold and let $K\subseteq X$ be metrically closed such that the complete geodesic lines determined by any two points of $K$ lie in $K$. Let $f: K \times K \rightarrow \R$ be such that $(A2)$, $(A4)$ and $(A5)$ hold. Fix $\alpha, \rho \in [0,1)$ and $x_0, x_{-1} \in K$ as well as $\{ \beta_{k}\}_k \subseteq (0,\infty)$, $\{ \alpha_{k}\}_k \subseteq [0,\alpha]$ and $\{\rho_k\}_k \subseteq [1-\rho,1+\rho]$ such that $(C1)^u$ and $(C2)$ hold. Assume that \eqref{RIPPA} is well-defined and let $\{x_k\}_k, \{y_k\}_k$ and $\{z_k\}_k$ be its generated sequences. Suppose further that $S(K,f) \neq \emptyset$. If 
\[
\sum_{k=0}^\infty\alpha_k d^2(x_k,x_{k-1})<+\infty,
\]
then the sequences $\{x_k\}_k$, $\{y_k\}_k$ and $\{z_k\}_k$ converge to the (unique) equilibrium point $x^{*}$ of $f$. Further, we have the following rate of convergence: If $M: (0,\infty) \rightarrow \N $ and $S \in (0,\infty)$ are such that $M$ is decreasing and
\[
\sum_{k = 0}^{\infty}\alpha_kd^2(x_k, x_{k-1}) \leq  S\text{ as well as }\forall \varepsilon > 0\left(\sum_{k = M(\varepsilon)}^{\infty}\alpha_kd^2(x_k, x_{k-1}) < \varepsilon\right),
\]
then
\begin{equation*}
\forall \varepsilon > 0\ \forall k \geq \left(\left \lceil \frac{C'}{\varepsilon^2} \right \rceil+ M(E' \varepsilon^2) + 2\right) \left(d(x_k, x^{*}), d(y_k, x^{*}), d(z_k, x^{*})< \varepsilon\right),
\end{equation*}
where 
\[
C' := \mathrm{max}\left\{1, \frac{(2\gamma-8\eta)^2}{\eta^3 \theta (1-\alpha)}\right\}\cdot\frac{64 (b+2S)(2\gamma-8\eta)(3\gamma-16\eta)^2}{\eta^3\theta(1-\alpha)^2(3\gamma-28\eta)}
\]
and
\[
E' := \mathrm{min}\left\{1, \frac{\eta^3 \theta (1-\alpha)}{32 (2\gamma-8\eta)^2}\right\}\cdot\frac{(1-\alpha)}{32} 
\]
with $b \geq d^2(x_0,x^{*})$.
\end{theorem}

\begin{proof}
The proof proceeds in exactly the same way as the proof of Theorem \ref{THEOREM3}, with the only difference that we refer to Lemma \ref{inertial_approx_fp_N} instead of Lemma \ref{inertial_approx_fp} and Lemma \ref{inertial_sum_good_N} instead of Lemma \ref{inertial_sum_good}. We in that way only provide a sketch. Write
\[
T' := \frac{64^2(b + 2S) ((1+\theta (\gamma-8\eta))(2\gamma - 8\eta)(3\gamma - 16 \eta)^2}{16 \eta^2 (1-\alpha)^2 (3\gamma - 28 \eta)^2 \theta}.
\]
Given $\varepsilon >0$, using Lemma \ref{inertial_approx_fp_N}, take
\[
M((1-\alpha)\varepsilon^2/32)  \leq k \leq \left \lceil\frac{T'}{\varepsilon^2}\right \rceil + M((1-\alpha)\varepsilon^2/32) +1
\]
such that
\[
d(z_k,y_k) <\frac{\sqrt{1-\alpha}}{64}\frac{3\gamma - 28\eta}{ (3\gamma-16\eta)} \varepsilon < \frac{\sqrt{1-\alpha}}{32}\varepsilon.
\]
Proceeding similarly as in the proof of Theorem \ref{THEOREM3}, one obtains
\[
\forall n  \geq \left \lceil \frac{T'}{\varepsilon^2}\right \rceil + M((1-\alpha)\varepsilon^2/32)  +2 \left( d(x_n, x^{*})< \frac{\varepsilon}{2}\right).
\]
Again, as $M: (0,\infty) \rightarrow \N$ is decreasing, Lemma \ref{inertial_sum_good_N} combined with the above yields
\[
d(z_n, y_n) < \sqrt{\frac{(1+\theta(\gamma-8\eta))(2\gamma-8\eta)}{16\eta^2 \theta (1-\alpha)}}d(x_j ,x^{*}) + \frac{\varepsilon}{16}.
\]
for any $n\geq j$ with
\[
j \geq\left \lceil \frac{T'}{\varepsilon^2} \frac{8^2(1+\theta(\gamma-8\eta))(2\gamma-8\eta)}{16 \eta^2 \theta(1-\alpha)}\right \rceil+ M\left(\frac{\eta^2 \theta (1-\alpha)^2\varepsilon^2}{128(1+\theta(\gamma-8\eta))(2\gamma-8\eta)}\right) +2.
\]
For such an $n$, similar to before, we get
\[
d(x_n, x^{*}) < \frac{\varepsilon^2}{16 \sqrt{\frac{(1+\theta(\gamma-8\eta))(2\gamma-8\eta)}{16\eta^2 \theta (1-\alpha)}}}
\]
Combining the above yields $d(z_n, y_n) < \frac{\varepsilon}{8}$ for such an $n$. One now obtains the result in the same way as in the proof of Theorem \ref{THEOREM3}, using various (tedious, but trivial) simplifications along the way.
\end{proof}

Remarks \ref{rem:meta} -- \ref{rem:quasiFejer} also apply in this context of $(C1)^u$, that is to Theorem \ref{THEOREM3a}.\\

\noindent{\textbf{Acknowledgments:}} This paper is a revised version of parts of the master thesis \cite{Despres2026} of the first author, written under the supervision of Prof.\ Dr.\ Ulrich Kohlenbach at TU Darmstadt and co-supervised by the second author. Both authors want to thank Prof.\ Kohlenbach for many insightful comments on the topic of the thesis \cite{Despres2026}. The authors also want to thank Sorin-Mihai Grad and Felipe Lara for helpful conversations on and around the topic of this paper.

\bibliographystyle{plain}
\bibliography{ref}

\end{document}